\documentclass[final]{siamltex}
\usepackage{hyperref}
\usepackage{amsmath,amssymb}%,amsthm}
\usepackage{graphicx}

\newcommand{\complex}{\mathbb{C}}
\newcommand{\real}{\mathbb{R}}
\newcommand{\nat}{\mathbb{N}}
\renewcommand{\Re}{\mbox{Re}\,}
\renewcommand{\Im}{\mbox{Im}\,}

\newcommand{\Mb}[1]{\left[{#1}\right]}
\newcommand{\Mcb}[1]{\left\{{#1}\right\}}

\newcommand{\Pq}{Q}
\newcommand{\Wq}{G}
\newcommand{\Ps}{P}
\newcommand{\Ws}{F}
\newcommand{\fq}{g}
\newcommand{\fs}{f}

\newcommand{\DG}{\nabla}

\newcommand{\cF}{\mathcal{F}}
\newcommand{\cG}{\mathcal{G}}
\newcommand{\cQ}{\mathcal{Q}}

\DeclareMathOperator{\expect}{\mathbb{E}}
\DeclareMathOperator{\proba}{\mathbb{P}}

\DeclareMathOperator{\linspan}{span}
\DeclareMathOperator{\nullspace}{\mathcal{N}}
\DeclareMathOperator{\colspace}{\mathcal{R}}

\newtheorem{example}[theorem]{Example}
\newtheorem{remark}[theorem]{Remark}

\title{On the solvability of the discrete\\ conductivity and
Schr\"odinger inverse problems}
\author{Justin Boyer\footnotemark[2]\ \footnotemark[4]
\and Jack J. Garzella\footnotemark[3]
\and Fernando Guevara Vasquez\footnotemark[2]\ \footnotemark[4]}

\begin{document}
\maketitle

\renewcommand{\thefootnote}{\fnsymbol{footnote}}
\footnotetext[2]{Mathematics dept., University of Utah, 155 S 1400 E RM
233, Salt Lake City UT 84112-0090.}
\footnotetext[3]{Juan Diego Catholic High School, 300 E 11800 S, Draper,
Utah 84020.}
\footnotetext[4]{Supported by the National Science Foundation
grant DMS-1411577.}
%\slugger{siap}{xxxx}{xx}{x}{x--x}
\begin{abstract}
We study the uniqueness question for two inverse problems on graphs.
Both problems consist in finding (possibly complex) edge or nodal based
quantities from boundary measurements of solutions to the Dirichlet
problem associated with a weighted graph Laplacian plus a diagonal
perturbation. The weights can be thought of as a discrete conductivity
and the diagonal perturbation as a discrete Schr\"odinger potential. We
use a discrete analogue to the complex geometric optics approach to show
that if the linearized problem is solvable about some conductivity (or
Schr\"odinger potential) then the linearized problem is solvable for
almost all conductivities (or Schr\"odinger potentials) in a suitable
set. We show that the conductivities (or Schr\"odinger potentials) in a
certain set are determined uniquely by boundary data, except on a zero
measure set. This criterion for solvability is used in a statistical
study of graphs where the conductivity or Schr\"odinger inverse problem
is solvable.
\end{abstract}

\pagestyle{myheadings}
\thispagestyle{plain}
\markboth{J. BOYER, J.~J. GARZELLA AND F. GUEVARA VASQUEZ}{SOLVABILITY OF DISCRETE INVERSE PROBLEMS}

%\tableofcontents

%%%%%%%%%%%%%%%%%%%%%%%%%%%%%%%%%%%%%%%%%%%%%%%%%%%%%%%%%%%%%%%%%%%%%%%%
\section{Introduction}

We consider the problem of finding nodal and edge quantities in a graph
from measurements made at a few nodes that are accessible.  To give a
concrete example consider the electrical circuit given in
figure~\ref{fig:circuit}. The only nodes we have access to are in white
and are called {\em boundary} nodes, while the nodes in black are
inaccessible and are called {\em interior} nodes. Each edge $e$ between
two interior nodes in the graph corresponds to an electrical component
with {\em complex conductivity} or {\em admittivity} $\gamma(e)$ (the
reciprocal of the {\em impedivity}). A complex valued voltage difference
$V$ across such a component is related to the current $I$ traversing the
component by Ohm's law $I = \gamma(e) V$.  Moreover each internal node
$v$ is joined to the ground (a zero voltage internal node) by a complex
conductivity $q(v)$ that we call Schr\"odinger potential or leak. We
consider the solvability question for the following two discrete inverse
problems, i.e. whether the data uniquely determines the unknown.
\begin{itemize}
 \item[] {\bf Discrete inverse conductivity problem}:
 Assuming $q=0$, find the admittances $\gamma$
 from electrical measurements at the boundary nodes.

 \item[] {\bf Discrete inverse Schr\"odinger problem}: 
 Assuming $\gamma$ is known, find the leaks $q$ from electrical
 measurements at the boundary nodes.
\end{itemize}

For the conductivity inverse problem we show that when $\Re \gamma > 0$,
the solvability question depends only on the topology of the underlying
graph (i.e. the circuit layout). We give also an easy criterion to
identify the graphs on which the conductivity can be recovered. To
explain this criterion, consider the forward map $\cF$ that to $\gamma$
associates the measurements at the boundary nodes (which are explained
in detail in \S\ref{sec:dsc}). We show that if for some conductivity
$\rho$ with $\Re \rho>0$ the Jacobian $D_\gamma \cF[\rho]$ is injective
then we have uniqueness for almost all other pairs of conductivities
$\gamma_1,\gamma_2$ with positive real part, i.e. $\cF(\gamma_1) =
\cF(\gamma_2)$ implies $\gamma_1 = \gamma_2$ except for a set of measure
zero. Thus we can only guarantee that the equivalence classes for the
equivalence relation $\cF(\gamma_1) = \cF(\gamma_2)$ are of measure zero.
In particular this means that there may be lower dimensional manifolds
(such as segments) of conductivities that share the same data
$\cF(\gamma)$, but these sets must have an empty interior.

We prove a similar result for the Schr\"odinger inverse problem when
$\Re \gamma >0$. Let $\cQ$ be the forward map that to the Schr\"odinger
potential $q$ associates the measurements at the boundary nodes. Then if
for some leak $p$ with $\Re p \geq 0$ the Jacobian $D_q \cQ[p]$ is
injective, then we have uniqueness for almost all other pairs of
Schr\"odinger potentials with non-negative real part. ($\Re q$
can be negative up to a certain extent dictated by $\gamma$).

\begin{figure}
 \begin{center}
 \includegraphics[width=0.3\textwidth]{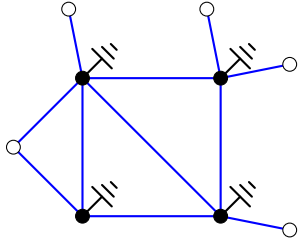}
 \end{center}
 \caption{A circuit with boundary nodes in white and interior nodes in
 black.  Each interior node is connected to the ground, i.e. a single
 node with constant zero potential. The edges in blue correspond to a
 conductivity $\gamma$ and the edges in black to a Schr\"odinger
 potential $q$.} 
 \label{fig:circuit}
\end{figure}

%%%%%%%%%%%%%%%%%%%%%%%%%%%%%%%%%%%%%%%%%%%%%%%%%%%%%%%%%%%%%%%%%%%%%%%%
\subsection{Applications}
The problem of finding the conductivities in a graph arises in
applications such as geophysical exploration \cite{Dines:1981:AEC} and
medical imaging in a modality called electrical impedance tomography or
EIT, for a review see \cite{Borcea:2012:RNA}. The methods described in
\cite{Borcea:2012:RNA} are limited to 2D setups and real conductivities
because they rely on results for the discrete inverse conductivity
problem on graphs that were only known to hold for positive
conductivities and for planar graphs. However, data from EIT is usually
not static, and is usually taken for different frequencies because
admittivity can help distinguish between different types of tissue (see
e.g. the reviews \cite{Cheney:1999:EIT,Borcea:2002:EIT}). With the
results presented here, we expect that the methods in
\cite{Borcea:2012:RNA} can be generalized to handling 3D setups, complex
conductivities (admittivities) and the Schr\"odinger problem. As far as
the continuum Schr\"odinger problem is concerned, some of the techniques
in \cite{Borcea:2012:RNA} are used in the forthcoming
\cite{Borcea:2015:NSI} to design a numerical method for solving the
two-dimension Schr\"odinger problem in the (real) absorptive case which
arises in e.g. diffuse optical tomography (see e.g. the review
\cite{Arridge:1999:OTM}).

%%%%%%%%%%%%%%%%%%%%%%%%%%%%%%%%%%%%%%%%%%%%%%%%%%%%%%%%%%%%%%%%%%%%%%%%
\subsection{Contents} \label{sec:contents}
We start in \S\ref{sec:related} by giving a review of available
uniqueness results for inverse problems on graphs. Then in
\S\ref{sec:cgo} we explain the parallels between our method and
the complex geometric optics approach (in the continuum). This section
is self-contained and its purpose is to motivate our approach.
The forward and inverse
discrete problems that we consider are defined in \S\ref{sec:dsc}.
The uniqueness result for the conductivity problem is in
\S\ref{sec:cond} and for the Schr\"odinger problem in
\S\ref{sec:schroe}. Numerical experiments, including a statistical study of graphs
where either of these inverse problems can be solved are in \S\ref{sec:numerics}.
Interestingly, this study reveals that graphs on which the Schr\"odinger
problem is solvable are much more common than graphs where the
conductivity problem is solvable.  We conclude
with a discussion in \S\ref{sec:discussion}. Because of their similarity
to their conductivity problem analogues, we defer the proof of some
results for the Schr\"odinger problem to Appendix~\ref{sec:qbur}.

%%%%%%%%%%%%%%%%%%%%%%%%%%%%%%%%%%%%%%%%%%%%%%%%%%%%%%%%%%%%%%%%%%%%%%%%
\section{Related work}\label{sec:related}
We review uniqueness results for the discrete inverse conductivity
problem (\S\ref{sec:rc}), for the discrete inverse Schr\"odinger
problem (\S\ref{sec:rq}) and finally some results on infinite
lattices (\S\ref{sec:lattices}).

%%%%%%%%%%%%%%%%%%%%%%%%%%%%%%%%%%%%%%%%%%%%%%%%%%%%%%%%%%%%%%%%%%%%%%%%
\subsection{Discrete inverse conductivity problem}
\label{sec:rc}
The first uniqueness result for the discrete inverse conductivity
problem that we are aware of is the constructive algorithm of Curtis and Morrow
\cite{Curtis:1990:DRN} to find a positive real conductivity in a
rectangular graph. Then Curtis, Mooers and Morrow \cite{Curtis:1994:FCC}
gave a constructive algorithm to find the positive real conductivity
in certain circular planar networks. This result was generalized
independently by Colin de Verdi\`ere \cite{Colin:1994:REP} and by
Curtis, Ingerman and Morrow \cite{Curtis:1998:CPG} to encompass positive
real conductivities on circular planar graphs, i.e. all the graphs that
can be embedded in a plane without edge crossings and for which the
boundary (or accessible) nodes lie on a circle. The condition for
recoverability is a condition on the connectedness of the graph and does
not depend on the conductivity (see also
\cite{Colin:1996:REP,Colin:1998:SG,curtis:2000:IPE}).  

Other results by Chung and Berenstein \cite{Chung:2005:HFI} and Chung
\cite{Chung:2010:IRE} are not limited to circular planar graphs, but
assume monotonicity and positive real conductivities. These results are
discrete analogous of the uniqueness result of Alessandrini
\cite{Alessandrini:1989:RPB} for the continuum conductivity problem (see
also \cite[\S 4.3]{Isakov:2006:IPD}). In a nutshell, the result of Chung
\cite{Chung:2005:HFI} guarantees that if two conductivities $\gamma_1
\leq \gamma_2$ (the inequality being componentwise) agree in a
neighborhood of the boundary nodes and one measurement of the potential
and the corresponding current at the boundary is identical for both
conductivities then $\gamma_1 = \gamma_2$. The result
\cite{Chung:2010:IRE} is an extension to other kinds of measurements.

We point also to the work by Lam and Pylyavskyy \cite{Lam:2012:IPC},
which uses combinatorics to study networks that can be embedded in a
cylinder (i.e.  the nodes and edges lie on the cylinder's surface with
no edges crossing) and with boundary nodes at the two ends of the
cylinder.  They show that positive real conductivities cannot be
uniquely determined from boundary data.

To our knowledge, our uniqueness result is the first that deals with
complex conductivities and that is not limited to circular planar
graphs. However our result is weaker than that in
\cite{Colin:1994:REP,Curtis:1998:CPG} in the sense that we do not show
uniqueness for all conductivities with positive real part, but we do
show that the set of pairs of conductivities $\gamma_1,\gamma_2$ that we
cannot tell apart from boundary data is a set of measure zero.

%%%%%%%%%%%%%%%%%%%%%%%%%%%%%%%%%%%%%%%%%%%%%%%%%%%%%%%%%%%%%%%%%%%%%%%%
\subsection{Discrete inverse Schr\"odinger problem}
\label{sec:rq}
The inverse Schr\"odinger problem on circular planar graphs has been
studied by Ara\'uz, Carmona and Encinas
\cite{Arauz:2014:DRM,Arauz:2015:OPB, Arauz:2015:DRM}, and they also give
conditions for which real Schr\"odinger potentials in a certain class
(that makes them essentially equivalent to conductivities) can be
determined from boundary data. Although our result is stronger in the
sense that we allow for complex Schr\"odinger potentials and topologies
that are not necessarily planar, our result is also weaker in the sense
that we only have uniqueness up to a zero measure set of pairs of
Schr\"odinger potentials $q_1,q_2$ that are indistinguishable from
boundary data.

%%%%%%%%%%%%%%%%%%%%%%%%%%%%%%%%%%%%%%%%%%%%%%%%%%%%%%%%%%%%%%%%%%%%%%%%
\subsection{Infinite lattices}
\label{sec:lattices}
We also note that there are related uniqueness results for infinite
graphs (i.e. lattices) for both the inverse Schr\"odinger problem
\cite{Morioka:2011:IBV, Isozaki:2012:IPT, Ando:2013:IST} and
conductivity problem \cite{Ervedoza:2011:USE}. The latter being perhaps
the closest to our results because it uses the complex geometric optics
approach.

%%%%%%%%%%%%%%%%%%%%%%%%%%%%%%%%%%%%%%%%%%%%%%%%%%%%%%%%%%%%%%%%%%%%%%%%
\section{Relation with the complex geometric optics approach}
\label{sec:cgo}
We give here a brief summary of the uniqueness proof for the inverse
Schr\"odinger problem by Sylvester and Uhlmann
\cite{Sylvester:1987:GUT}. The intention of this summary is not to be
complete or general, but to highlight the steps in the proof that have a
discrete analogue in our argument. For reviews focussed on the complex
geometric optics (CGO) approach for the inverse conductivity problem see
e.g.  \cite{Uhlmann:2009:EIT,Salo:2008:CP}, for other inverse problems
see e.g. \cite{Uhlmann:1999:DIP}. Finally note that this is the only
section where we use functions defined on a continuum, elsewhere
functions are defined on a finite sets. Hence the notation is exclusive
of this section.

%%%%%%%%%%%%%%%%%%%%%%%%%%%%%%%%%%%%%%%%%%%%%%%%%%%%%%%%%%%%%%%%%%%%%%%%
\subsection{The continuum inverse Schr\"odinger problem}
Let $\Omega \subset \real^d$ be a domain with smooth boundary
$\partial\Omega$ in dimension $d\geq 3$ and let $q \in L^2(\Omega)$ be
a, possibly complex valued, function that we shall call Schr\"odinger
potential. The Dirichlet problem for the Schr\"odinger equation is given
the Dirichlet data $f \in H^{1/2}(\partial\Omega)$, find $u \in
H^1(\Omega)$ such that
\begin{equation}
\label{eq:schroe}
 - \Delta u + q u = 0~\text{in $\Omega$},~\text{and}~
 u  = f ~\text{on $\partial\Omega$}.
\end{equation}
The Dirichlet to Neumann map is the linear mapping $\Lambda_q : H^{1/2}
(\partial\Omega) \to H^{-1/2}(\partial\Omega)$ defined for $f \in
H^{1/2}(\partial\Omega)$ by
\begin{equation}
 \Lambda_q f = n \cdot \nabla u,
\end{equation}
where $u$ solves \eqref{eq:schroe} with Dirichlet boundary condition $f$
and $n$ is the unit normal to $\partial\Omega$ pointing outwards. We
assume here that the Schr\"odinger potential $q$ is such that the
Dirichlet problem \eqref{eq:schroe} admits a unique solution for all
boundary excitations $f\in H^{1/2}(\partial\Omega)$. The {\em inverse
Schr\"odinger problem} is to find $q$ from the Dirichlet to Neumann map
$\Lambda_q$.

%%%%%%%%%%%%%%%%%%%%%%%%%%%%%%%%%%%%%%%%%%%%%%%%%%%%%%%%%%%%%%%%%%%%%%%%
\subsection{An interior identity}
The first step in the uniqueness proof of Sylvester and Uhlmann
\cite{Sylvester:1987:GUT} for the Schr\"odinger inverse problem is to
show an identity that relates a difference in boundary data to a
difference in Schr\"odinger potentials times products of solutions to
the Dirichlet problem:
\begin{equation}
\label{eq:ii}
\int_{\partial\Omega} (\Lambda_{q_1} - \Lambda_{q_2}) (f_1)  f_2  dS =
\int_\Omega (q_1 - q_2) u_1 u_2 dx,
\end{equation}
where $u_j$ solves the Dirichlet problem \eqref{eq:schroe} with boundary
data $f_j \in H^{1/2}(\partial\Omega)$, for $j=1,2$. 

%%%%%%%%%%%%%%%%%%%%%%%%%%%%%%%%%%%%%%%%%%%%%%%%%%%%%%%%%%%%%%%%%%%%%%%%
\subsection{Density of products of solutions}
The second step is to prove that products of solutions to the Dirichlet
problem \eqref{eq:schroe} are dense in some appropriate space, say
$L^2(\Omega)$. If this were true and we had two Schr\"odinger potentials
$q_1$ and $q_2$ for which $\Lambda_{q_1} = \Lambda_{q_2}$ we could use
\eqref{eq:ii} to conclude that $\overline{q}_1 - \overline{q}_2$ is in
the orthogonal of a set that is dense in $L^2(\Omega)$.  Therefore we
get $q_1 = q_2$ and uniqueness.

Calder\'on \cite{Calderon:1980:IBV} proved that for the harmonic case
($q=0$) there is a family of complex exponential harmonic functions
(hence the CGO name) such that for any $\xi \in \real^d$, one can pick
two members $v_1$, $v_2$ in the family so that $v_1 v_2 = e^{i x \cdot
\xi}$. Since $\xi$ is arbitrary, this is of course dense in
$L^2(\Omega)$ by the Fourier transform.  Crucially these members can be
chosen with arbitrarily high frequencies.  Since \eqref{eq:schroe} is
the Laplacian plus a lower order term, $v_1$ and $v_2$ plus appropriate
corrections solve \eqref{eq:schroe}, and these corrections vanish in the
high frequency limit. Hence high frequency asymptotic estimates show
that if products of solutions to the Dirichlet problem with $q=0$ are
dense in $L^2(\Omega)$, then products of solutions for any other $q$ are
dense in $L^2(\Omega)$.

%%%%%%%%%%%%%%%%%%%%%%%%%%%%%%%%%%%%%%%%%%%%%%%%%%%%%%%%%%%%%%%%%%%%%%%%
\subsection{Parallels with the discrete setting}
In a finite graph we work with functions defined on finite sets, so we
do not have such high frequency asymptotics. The concept that replaces
this is that of analytic continuation for functions of several complex
variables (see e.g. \cite{Gunning:1965:AFS}). Also in finite dimensions,
saying that a family of vectors is dense in a space, simply means that
the family spans the space.

%%%%%%%%%%%%%%%%%%%%%%%%%%%%%%%%%%%%%%%%%%%%%%%%%%%%%%%%%%%%%%%%%%%%%%%%
\section{The discrete Schr\"odinger and conductivity problems}
\label{sec:dsc}

We start in \S\ref{sec:prelim} with some preliminaries, then in
\S\ref{sec:lap} we define the weighted graph Laplacian. In
\S\ref{sec:dir} we give sufficient conditions on the conductivity
and the Schr\"odinger potential for the Dirichlet problem to have a
unique solution no matter what the prescribed value at the boundary
nodes is. In \S\ref{sec:ip} we formulate the inverse
Schr\"odinger and conductivity problems. Then in \S\ref{sec:green} we
give a discrete analogue to the Green identities.

%%%%%%%%%%%%%%%%%%%%%%%%%%%%%%%%%%%%%%%%%%%%%%%%%%%%%%%%%%%%%%%%%%%%%%%%
\subsection{Preliminaries}
\label{sec:prelim}
We consider graphs $\cG = (V,E)$, where $V$ is a finite vertex set and $E
\subset \{ \{i,j\} ~ | ~ i,j \in V,~i\neq j\}$ is the edge set. The vertex
set $V$ is partitioned into boundary nodes $B$ and interior nodes $I$.
We use the set theory notation for functions, e.g. $f \in \complex^V$ is a
function $f: V \to \complex$. Upon fixing an ordering of the vertices, $f$
can be identified to a vector in $\complex^{|V|}$, that in a slight
abuse of notation is also denoted by $f$. In the same way, we identify a
linear operator $A: \complex^X \to \complex^Y$, where $X$ and $Y$
are finite sets (e.g. $X=E$ and $Y=V$), to a $|X| \times |Y|$
complex matrix which is also denoted by $A$. 

%%%%%%%%%%%%%%%%%%%%%%%%%%%%%%%%%%%%%%%%%%%%%%%%%%%%%%%%%%%%%%%%%%%%%%%%
\subsection{The weighted graph Laplacian}
\label{sec:lap}
The {\em discrete gradient}
is the linear map $\DG : \complex^V \to \complex^E$ such that
\begin{equation}
 (\DG u)(\{i,j\}) = u(i) - u(j),~\mbox{for $u \in \complex^V$ and $\{i,j\}
 \in E$.}
\end{equation}
This definition depends on the ordering of vertices for an edge, but as
long as it is fixed once and for all it is inessential to the following
discussion.

Given a {\em discrete conductivity}
$\gamma \in \complex^E$, the {\em weighted graph
Laplacian} $L_\gamma : \complex^V \to \complex^V$ by (see e.g.
\cite{Chung:1997:SGT})
\begin{equation}
 L_\gamma u = \DG^* [ \gamma \odot (\DG u) ] = \sum_{\{i,j\} \in E}
 \gamma(\{i,j\}) (u(i) - u(j)), ~\mbox{for $u \in \complex^V$.}
 \label{eq:lap}
\end{equation}
Here $\DG^* : \complex^E \to \complex^V$ is the adjoint of $\DG$. Since
the entries of $\DG$ are real, we have $\DG^* = \DG^T$. By $u
\odot v$ we mean the Hadamard product of two vectors, e.g. if $u,v \in
\complex^E$, $(u\odot v)(e) = u(e)v(e)$ for all $e \in E$. In matrix
notation we have $L_\gamma = \DG^* \diag(\gamma) \DG$.

%%%%%%%%%%%%%%%%%%%%%%%%%%%%%%%%%%%%%%%%%%%%%%%%%%%%%%%%%%%%%%%%%%%%%%%%
\subsection{The Dirichlet problem}
\label{sec:dir}
For given $\gamma \in \complex^E$ and $q \in \complex^I$, the {\em
Dirichlet problem for $\gamma,q$} consists of finding the interior values $u_I \in
\complex^I$ from the boundary values $u_B \in \complex^B$ such that
\begin{equation}
 (L_\gamma)_{IB} u_B + (L_\gamma)_{II} u_I + q \odot u_I = 0.
 \label{eq:int}
\end{equation}
Here $(L_\gamma)_{IB}$ means we restrict the rows of the Laplacian
$L_\gamma$ to the index set $I$ and the columns to the index set $B$.  A
solution $u \in \complex^V$ to the Dirichlet problem for $\gamma,q$ is
said to be $\gamma,q$ {\em harmonic}. We say the Dirichlet problem for
$\gamma,q$ is {\em well-posed} when the matrix $(L_\gamma)_{II} +
\diag(q)$ is invertible. This means that the interior values $u_I \in
\complex^I$ are uniquely determined by the boundary values $u_B \in
\complex^B$. The following theorem gives conditions guaranteeing the
$\gamma,q$ Dirichlet problem is well-posed. To write these conditions
concisely, we denote by $\complex_\zeta \equiv \{ z \in \complex ~ | ~
\Re z > \zeta \}$ the open region to the right of the line $\Re z =
\zeta$ in the complex plane, where $\zeta \in \real$. In what follows it
is convenient to consider the subgraph $\cG_I = (I,E_I)$ of $\cG$
restricted to the interior nodes and with edge set 
\begin{equation}
 E_I \equiv \{ \{i,j\} \in E ~|~ i,j\in I\}.
 \label{eq:ei}
\end{equation}

\begin{theorem}
\label{thm:dir}
Assuming $\cG$ and $\cG_I$ are connected graphs, for all $\gamma \in
\complex_0^E$ there is a real number $\zeta<0$ such that  the Dirichlet
problem for $\gamma,q$ is well-posed for $q \in \complex_\zeta^I$.
\end{theorem}

We call {\em restricted Laplacian}, the Laplacian $L_{\gamma_I}$ on
$\cG_I$ with weights $\gamma_I \equiv \gamma_{E_I}$. The restricted
Laplacian and the $II$ block of the full Laplacian differ only by a
diagonal term $\mu \in \complex^I$, i.e.
\begin{equation}
 (L_\gamma)_{II} = L_{\gamma_I} + \diag(\mu),
 \label{eq:rlap}
\end{equation}
where the entries of $\mu$ are 
\begin{equation} 
 \mu(j) = \sum_{i \in B, ~\{i,j\} \in E} \gamma(\{i,j\}), ~\text{for $j\in I$}.
 \label{eq:mu}
\end{equation}
If no entry of $\gamma$ is zero, the support of $\mu$ (i.e. the nodes
$i$ for which $\mu(i)$ is non-zero) consists of the interior nodes that
are connected by an edge to the boundary nodes, i.e. the nodes in the
set
\begin{equation}
 J = \{ i \in I ~|~ \{i,j\} \in E ~\text{for some $j\in B$} \}.
 \label{eq:j}
\end{equation}

\begin{proof}[Proof of theorem~\ref{thm:dir}]
Since $\gamma$ and $q$ are complex it is convenient to introduce the
notation $\gamma = \gamma' + i \gamma''$, where $\gamma' \equiv \Re \gamma$
and $\gamma'' \equiv \Im \gamma$ and similarly for $q$ and $\mu$ (as
defined in \eqref{eq:mu}).  We show that under the hypothesis of
the theorem, the matrix $(L_{\gamma})_{II} + \diag(q)$ is invertible.
We do this by showing that the field of values (see e.g.
\cite{Horn:2013:MA}) of $(L_{\gamma})_{II} + \diag(q)$, i.e. the
region $W$ of the complex plane defined by
\begin{equation}
 W \equiv \Mcb{ u^* [(L_{\gamma})_{II} + \diag(q)] u ~|~ u \in \complex^I
 ~\text{with}~\|u\|^2=1},
\end{equation}
is contained in $\complex_0$. Since the spectrum is contained in the field of values of a
matrix, this would mean that $0$ cannot be an eigenvalue of
$(L_{\gamma})_{II} + \diag(q)$, which is then invertible.
The proof is divided in three steps.

 {\bf Step 1. $\Re (L_{\gamma})_{II}$ is Hermitian positive definite.}
 Indeed we have
 \[
  \Re (L_{\gamma})_{II} = L_{\gamma'_I} + \diag(\mu').
 \]
 Since $\gamma \in \complex_0^E$, we clearly have $\gamma'_I>0$,
 $\mu'_J > 0$ and $\mu'_{I-J} = 0$.  Since the graph $\cG_I$ is
 connected, the restricted Laplacian $L_{\gamma'_I}$ must be Hermitian
 positive semidefinite with a one-dimensional nullspace spanned by the
 constant vector $[1,\ldots,1]^T \in \complex^I$ (see e.g.
 \cite{Chung:1997:SGT} or \cite{Colin:1998:SG}). Hence we have $u^*
 [L_{\gamma'_I} + \diag(\mu')] u \geq 0$ for all $u \in \complex^I$.
 Moreover if $u^* [L_{\gamma'_I} + \diag(\mu')] u = 0$, then $u =
 [a,\ldots,a]^T \in \complex^I$ for some $a \in \complex$. But then we
 must also have 
 \[ 
  0 = |a|^2 [1,\ldots,1] \diag(\mu') [1,\ldots,1]^T = |a|^2 \sum_{j\in J}
  \mu'(j).
 \]
 Since $\mu'_J > 0$, the sum above is positive. Thus $|a|^2=0$ and $u=0$, which
 proves the desired result. 
 
 {\bf Step 2. $\Re [ (L_{\gamma})_{II} + \diag(q) ]$ is positive
 definite}, for all $q \in \complex^I$ with  $q' > \zeta_\gamma$,
 where 
 \begin{equation}
  \zeta_\gamma \equiv -\lambda_{\min}((L_{\gamma'})_{II}),
  \label{eq:zetagamma}
 \end{equation}
 and $\lambda_{\min}(A)$ denotes the smallest eigenvalue (in
 magnitude) of a matrix $A$. By Step 1, $(L_{\gamma'})_{II}$ is
 Hermitian positive definite so $\zeta_\gamma<0$. To show the desired
 result, notice that
 \[ 
  \Re [ (L_{\gamma})_{II} + \diag(q) ] = (L_{\gamma'})_{II} +
  \diag(q'),
 \]
 is a Hermitian matrix, and for $u \in \complex^I$ with $\|u\|^2 =
 1$ its (real) Rayleigh quotient satisfies
 \[
  \begin{aligned}
  u^* [ (L_{\gamma'})_{II} + \diag(q') ] u & \geq
  \lambda_{\min}((L_{\gamma'})_{II}) \|u\|^2  + \sum_{i\in I} q'(i) |u(i)|^2 \\
  & > \lambda_{\min}((L_{\gamma'})_{II}) \|u\|^2 +  \zeta_\gamma \|u\|^2 = 0.
  \end{aligned}
 \]

 {\bf Step 3. The field of values $W$ is contained in $\complex_0$}. Let
 $u \in \complex^I$ with $\|u\|^2 = 1$. Then we can write
 \[
    u^* [(L_{\gamma})_{II} + \diag(q)] u = u^* \Re [(L_{\gamma})_{II} +
    \diag(q)] u + i u^* \Im [(L_{\gamma})_{II} + \diag(q)] u.
 \]
 Since $\Im [(L_{\gamma})_{II} + \diag(q)] = (L_{\gamma''})_{II} +
 \diag(q'')$
 is Hermitian, we clearly have that $ u^* \Im [(L_{\gamma})_{II} + \diag(q)] u \in \real$ and
 \[
  \Re(u^* [(L_{\gamma})_{II} + \diag(q)] u) = u^* \Re [(L_{\gamma})_{II} +
    \diag(q)] u > 0,
 \]
 by the result of Step 2. Thus any point in $W$ must have
 positive real part.
\end{proof}

%%%%%%%%%%%%%%%%%%%%%%%%%%%%%%%%%%%%%%%%%%%%%%%%%%%%%%%%%%%%%%%%%%%%%%%%
\subsection{The inverse Schr\"odinger and conductivity problems}
\label{sec:ip}
Provided the  Dirichlet problem for $\gamma,q$
is well-posed, we can define the {\em Dirichlet to Neumann map}
$\Lambda_{\gamma,q} : \complex^B \to \complex^B$ by the linear operator
\begin{equation}
 \Lambda_{\gamma,q} u_B = (L_\gamma)_{BB} u_B + (L_\gamma)_{BI} u_I,
 \label{eq:dtn1}
\end{equation}
where $u_B \in \complex^B$ and $u_I \in \complex^I$ is obtained from $u_B$ by solving
\eqref{eq:int}. The Dirichlet to Neumann map can be written as the Schur
complement of the block $(L_\gamma)_{II} + \diag(q)$ in the matrix
\begin{equation}
\begin{bmatrix}
 (L_\gamma)_{BB} & (L_\gamma)_{BI}\\
 (L_\gamma)_{IB} & (L_\gamma)_{II} + \diag(q)
\end{bmatrix},
\end{equation}
in other words:
\begin{equation}
 \Lambda_{\gamma,q} = (L_\gamma)_{BB} - (L_\gamma)_{BI} 
 [(L_\gamma)_{II} + \diag(q)]^{-1} (L_\gamma)_{IB}.
 \label{eq:dtn}
\end{equation}

The {\em inverse conductivity problem} is to find $\gamma \in
\complex^E$ from the Dirichlet to Neumann map $\Lambda_{\gamma,0}$,
assuming the graph $\cG=(V,E)$ and the boundary nodes $B$ are known.

The {\em inverse Schr\"odinger problem} is to find $q \in \complex^I$
from the Dirichlet to Neumann map $\Lambda_{\gamma,q}$, assuming
the graph $\cG=(V,E)$, the boundary nodes $B$, and the
conductivity $\gamma$ are known.

%%%%%%%%%%%%%%%%%%%%%%%%%%%%%%%%%%%%%%%%%%%%%%%%%%%%%%%%%%%%%%%%%%%%%%%%
\subsection{A discrete Green identity}
\label{sec:green}
Here is a discrete analogue of one of the Green identities that relates
a sum over boundary nodes to a sum over all the nodes.
\begin{lemma}[Green identity]\label{lem:gi} 
Let $\gamma$ and $q$ be such that the $\gamma,q$ Dirichlet problem is
well-posed.  Let $u,v \in \complex^V$, with $v$ being $\gamma,q$
harmonic. Then we have the following:
\begin{equation}
 u_B^T \Lambda_{\gamma,q} v_B = u^T L_\gamma v + u_I^T \diag(q) v_I.
\end{equation}
\end{lemma}
\begin{proof}
The right hand side can be written as
\begin{equation}
u^T L_\gamma v + u_I^T \diag(q) v_I = u_B^T(L_\gamma v)_B +
u_I^T(L_\gamma v + \widetilde{q} \odot v)_I  = u_B^T(L_\gamma v)_B,
\end{equation}
where $\widetilde{q}_B \equiv 0$ and $\widetilde{q}_I \equiv q$. The
last equality comes from $v$ being a solution to the Dirichlet problem
\eqref{eq:int} for $\gamma,q$. The desired result follows from the
definition of the Dirichlet to Neumann map \eqref{eq:dtn1} by noticing
that 
\[
 (L_\gamma v)_B = (L_\gamma)_{BB} v_B + (L_\gamma)_{BI} v_I =
 \Lambda_{\gamma,q} v_B,
\]
when $v$ is $\gamma,q$ harmonic.
\end{proof}

%%%%%%%%%%%%%%%%%%%%%%%%%%%%%%%%%%%%%%%%%%%%%%%%%%%%%%%%%%%%%%%%%%%%%%%%
\section{Solvability for the inverse conductivity problem}
\label{sec:cond}
We now show a uniqueness result for the inverse conductivity problem. We
start by showing an identity similar to the one used in
\cite{Calderon:1980:IBV,Sylvester:1987:GUT} that relates a difference in
boundary data to a difference in conductivities
(\S\ref{sec:sint}). This identity relates the uniqueness question
to studying the space spanned by products of (discrete) gradients of
solutions (\S\ref{sec:sprod}). The result has applications to Newton's
method (\S\ref{sec:newton}) and can be interpreted probabilistically
(\S\ref{sec:proba}). The main uniqueness result is proved in
\S\ref{sec:sbur}.

%%%%%%%%%%%%%%%%%%%%%%%%%%%%%%%%%%%%%%%%%%%%%%%%%%%%%%%%%%%%%%%%%%%%%%%%
\subsection{An interior identity}
\label{sec:sint}

The following lemma relates the difference in the boundary data for two
conductivities to the conductivity difference. This identity is a
discrete analogue to the one appearing in \cite{Sylvester:1987:GUT}.

\begin{lemma}[Interior identity for conductivities]
\label{lem:sint}
Let $\gamma_1,\gamma_2 \in \complex_0^E$, $u$ be $\gamma_1,0$
harmonic and $v$ be $\gamma_2,0$ harmonic. Then the following identity
holds
\begin{equation}
\label{eq:sint}
 u_B ^T ( \Lambda_{\gamma_1,0} - \Lambda_{\gamma_2,0} ) v_B = (\gamma_1
 - \gamma_2)^T \Mb{ (\DG u) \odot (\DG v)}.
\end{equation}
\end{lemma}
\begin{proof}
By using lemma~\ref{lem:gi} twice we get:
\[
  u_B^T \Lambda_{\gamma_1,0} v_B = u^T L_{\gamma_1,0} v, ~\text{and}~
  u_B^T \Lambda_{\gamma_2,0} v_B = u^T L_{\gamma_2,0} v.
\]
Subtracting the second equation from the first we get
\[
  u_B ^T ( \Lambda_{\gamma_1,0} - \Lambda_{\gamma_2,0} ) v_B = 
   u^T (L_{\gamma_1,0}  - L_{\gamma_2,0}) v
   = u^T \DG^T \diag(\gamma_1 - \gamma_2) \DG v,
\]
which gives the desired result.
\end{proof}

%%%%%%%%%%%%%%%%%%%%%%%%%%%%%%%%%%%%%%%%%%%%%%%%%%%%%%%%%%%%%%%%%%%%%%%%
\subsection{Uniqueness almost everywhere}
\label{sec:sprod}
Inspired by the complex geometric optics method
\cite{Calderon:1980:IBV,Sylvester:1987:GUT}, we would like to study the
subspaces of $\complex^E$
\begin{equation}
 \Ps(\gamma_1,\gamma_2) \equiv \linspan \Mcb{(\DG u)\odot(\DG v) ~|~
 \text{$u$ is $\gamma_1,0$ harmonic and $v$ is $\gamma_2,0$ harmonic}},
\end{equation}
for conductivities $\gamma_1,\gamma_2 \in \complex_0^E$. To see why this
subspace is important, assume there are two conductivities
$\gamma_1,\gamma_2$ with the same boundary data, i.e.
$\Lambda_{\gamma_1,0} = \Lambda_{\gamma_2,0}$, then lemma~\ref{lem:sint}
implies that $\overline{\gamma}_1 - \overline{\gamma}_2 \in
\Ps(\gamma_1,\gamma_2)^\perp$. Hence if we are so fortunate to have
$\Ps(\gamma_1,\gamma_2) = \complex^E$ we would conclude that $\gamma_1 =
\gamma_2$. 

\begin{theorem}[Uniqueness almost everywhere for conductivities]
\label{thm:bus}
If there is $\rho_1,\rho_2 \in \complex_0^E$ such that
$\Ps(\rho_1,\rho_2) = \complex^E$, then $\Ps(\gamma_1,\gamma_2) =
\complex^E$ for almost all $(\gamma_1,\gamma_2) \in \complex_0^E \times
\complex_0^E$.
\end{theorem}

The condition $\Ps(\gamma_1,\gamma_2) = \complex^E$ implies that
$\Lambda_{\gamma_1,0} \neq \Lambda_{\gamma_2,0}$ whenever $\gamma_1 \neq
\gamma_2$, i.e. wherever this condition holds we can distinguish two
conductivities from their DtN maps. The conclusion of
Theorem~\ref{thm:bus} says that conductivities are distinguishable
except for a zero measure set $Z$ of $(\gamma_1,\gamma_2)$ for which
$\Ps(\gamma_1,\gamma_2) \neq \complex^E$. What happens in $Z$ is
not so clear cut. The set $Z$ contains all $(\gamma_1,\gamma_2)$ for
which $\gamma_1 \neq \gamma_2$ but $\Lambda_{\gamma_1,0} =
\Lambda_{\gamma_2,0}$. However $Z$ may also contain other
$(\gamma_1,\gamma_2)$ for which $\Lambda_{\gamma_1,0} \neq
\Lambda_{\gamma_2,0}$.  Hence, theorem~\ref{thm:bus} guarantees that the
set $C \subset Z$ of $(\gamma_1,\gamma_2)$ for which
$\Lambda_{\gamma_1,0} = \Lambda_{\gamma_2,0}$ but $\gamma_1 \neq
\gamma_2$ is of measure zero in $\complex_0^E \times \complex_0^E$. 
We can refine the conclusion of theorem~\ref{thm:bus} by considering 
equivalence classes of conductivities for
the equivalence relation $\Lambda_{\gamma_1,0}=\Lambda_{\gamma_2,0}$.
\begin{corollary}\label{cor:sloc0}
 If there is $\rho_1,\rho_2 \in \complex_0^E$ such that
 $\Ps(\rho_1,\rho_2) = \complex^E$, then any equivalence class of
 conductivities in $\complex_0^E$ for the equivalence relation
 $\Lambda_{\gamma_1,0}=\Lambda_{\gamma_2,0}$ must be of measure zero in
 $\complex_0^E$.
\end{corollary}
\begin{proof}
 Since the hypothesis of theorem~\ref{thm:bus} holds, the set $Z$ of
 $(\gamma_1,\gamma_2)$ for which $\Ps(\gamma_1,\gamma_2) \neq
 \complex^E$ is of measure zero. Let $M$ be an equivalence class of
 conductivities that are indistinguishable from their DtN maps.  With
 $\diag(M\times  M) = \{ (\gamma,\gamma) | \gamma \in M \}$, we clearly
 have $M\times M - \diag(M \times M) \subset Z$. Hence $M\times M -
 \diag(M\times M)$ is of measure zero and so is $M\times M$.  Therefore
 $M$ is of measure zero.
\end{proof}

For another point of view, we consider the
linearization of the boundary data $\Lambda_{\gamma,0}$ with respect to
changes in the conductivity $\gamma$.
\begin{lemma}[Linearization for conductivities]
\label{lem:lins}
Let $\gamma \in \complex_0^E$. Then for sufficiently small $\delta\gamma
\in \complex^E$,
\begin{equation}
 u_B^T \Lambda_{\gamma+\delta\gamma,0} v_B = u_B^T \Lambda_{\gamma,0}
 v_B + (\delta\gamma)^T[ (\DG u) \odot (\DG v) ] + o(\delta\gamma),
\end{equation}
where $u=(u_B,u_I)$ and $v=(v_B,v_I)$ are $\gamma,0$ harmonic.
\end{lemma}
\begin{proof}
Use the interior identity \eqref{eq:sint} with $\gamma_1 = \gamma +
\epsilon\delta\gamma$ and $\gamma_2 = \gamma$, divide by $\epsilon$ and take the
limit as $\epsilon \to 0$.
\end{proof}

The mapping $\gamma \to \Lambda_{\gamma,0}$ is thus Fr\'echet
differentiable and it's Jacobian $D_\gamma \Lambda_{\gamma,0} :
\complex^E \to \complex^{B\times B}$ about $\gamma$ is 
\begin{equation}
 D_\gamma \Lambda_{\gamma,0} \delta \gamma = L_{BI} L_{II}^{-1} \DG^T
 \diag(\delta\gamma) \DG L_{II}^{-1} L_{IB},
\end{equation}
for $\delta\gamma \in \complex^E$ and where we omitted the subscript
$\gamma$ in the weighted graph Laplacian $L_\gamma$. We say the
linearized problem is {\em solvable} at $\gamma$, when the Jacobian about
$\gamma$ is injective, i.e. the nullspace $\nullspace(D_\gamma
\Lambda_{\gamma,0}) = \{0\}$. Solvability of the linearized inverse
problem about $\gamma$ is of course equivalent to the range $\colspace(D_\gamma
\Lambda_{\gamma,0}^*) = \complex^E = \Ps(\gamma,\gamma)$, where the last
equality comes from lemma~\ref{lem:lins} and the definition of the
Jacobian. Hence by taking $\rho_1=\rho_2$
in theorem~\ref{thm:bus} we immediately get the following corollary.
\begin{corollary}
\label{cor:sloc2}
If the linearized problem is solvable about some $\rho \in \complex_0^E$
then $\Lambda_{\gamma_1,0} = \Lambda_{\gamma_2,0}$ implies $\gamma_1
=\gamma_2$, for almost all $(\gamma_1,\gamma_2) \in \complex_0^E \times
\complex_0^E$.
\end{corollary}

The same proof technique for Theorem~\ref{thm:bur} can be used
to prove the following.
\begin{corollary}
\label{cor:sloc1}
If the linearized problem is solvable about some $\rho \in
\complex_0^E$ then it is solvable for almost all $\gamma \in
\complex_0^E$.
\end{corollary}

A consequence of Corollary~\ref{cor:sloc1} is that the solvability of
the linearized inverse problem does not depend on the actual
conductivity (except for a set of measure zero). This is without any
topological assumptions on the graph. This fact was already known for
circular planar graphs \cite{Curtis:1994:FCC,Curtis:1998:CPG}.

Corollary~\ref{cor:sloc1} suggests a simple test for checking whether
the conductivity problem is solvable in a graph. All we need to do is
check whether the Jacobian about some conductivity, say $\gamma = 1$
(the constant conductivity equal to one), has a trivial nullspace. Of
course, this may give a false negative if $\Ps(\gamma,\gamma) =
\complex^E$ for almost all $\gamma \in \complex_0^E$ but $\Ps(1,1) \neq
\complex^E$. However we have verified numerically that this is unlikely
to happen.

The following corollary shows that if we start at a conductivity for
which the linearized problem is solvable and go in any direction we may
encounter only finitely many conductivities that have the same DtN map
or where the problem is not solvable. The proof of this result requires
notation introduced for the proof of theorem~\ref{thm:bus} and is thus
presented in \S\ref{sec:sbur}.

\begin{corollary}\label{cor:newton}
If the linearized problem is solvable at $\gamma \in \complex^E_0$ then
along any direction $\delta\gamma\in \complex^E$ there are at most finitely
many $t\in \complex$ with $\Re(\gamma + t \delta\gamma) > 0$ for which
\begin{itemize}
 \item[i.] the linearized problem
 is not solvable at $\gamma + t \delta\gamma$.
 \item[ii.]  $\gamma$ and $\gamma+t\delta\gamma$ may have the same DtN
 map (i.e. $\Ps(\gamma,\gamma + t \delta\gamma) \neq \complex^E$).
\end{itemize}
\end{corollary}

\begin{remark}\label{rem:sreal}
As appears later in \S\ref{sec:sbur}, the proof of
theorem~\ref{thm:bus} hinges on complex analyticity. If we use real
analyticity instead, we can get results analogous to
theorem~\ref{thm:bus} and corollaries~\ref{cor:sloc0}, \ref{cor:sloc2},
\ref{cor:sloc1} and \ref{cor:newton}, but where $\complex$ is replaced by $\real$. The set
$\real_0^E$ is then the (open) positive orthant of $\real^E$.
\end{remark}

%%%%%%%%%%%%%%%%%%%%%%%%%%%%%%%%%%%%%%%%%%%%%%%%%%%%%%%%%%%%%%%%%%%%%%%%
\subsection{Application to Newton's method}
\label{sec:newton}
Corollary~\ref{cor:newton} is useful to show that the systems in
Newton's method applied to finding the conductors in a graph are very
likely to admit a unique solution (if only a finite number of iterations
are carried out). Newton's method applied to finding $\gamma$ from
$\Lambda_{\gamma,0}$ takes the form (see e.g.  \cite{Nocedal:2006:NO})
\begin{tabbing}
xxx \= xxx \= xxx\kill\\ {\bf Newton's method}\\
$\gamma^{(0)} =$ given\\
for $k=0,1,2,\ldots$\\
\> Find step $\delta\gamma^{(k)}$ s.t. $D_\gamma\Lambda_{\gamma^{(k)},0} \delta
\gamma^{(k)} = \Lambda_{\gamma^{(k)},0} - \Lambda_{\gamma,0}$\\
\> Choose step length $t_k>0$\\
\> Update $\gamma^{(k+1)} = \gamma^{(k)} + t_k \delta\gamma^{(k)}$\\
\end{tabbing}
where the
$\gamma^{(k)}$ are the iterates that hopefully converge to $\gamma$ and
$t_k>0$ is a parameter used to adjust the step length and ensure
feasibility ($\Re \gamma^{(k)} > 0$). If the linearized problem about
the $k-$th iterate $\gamma^{(k)}$ is solvable, then the linear system we
need to solve to find the $k-$th step admits a unique solution
$\delta\gamma^{(k)}$. Moreover the linearized
problem is solvable almost everywhere (by corollary~\ref{cor:sloc2}),
so we can expect that the linearized problem is solvable at the next
iterate $\gamma^{(k)}$. In fact corollary~\ref{cor:newton} guarantees
that up to finitely many exceptions, all choices of the next iterate
$\gamma^{(k+1)}$ are such that the linearized problem is solvable.

%%%%%%%%%%%%%%%%%%%%%%%%%%%%%%%%%%%%%%%%%%%%%%%%%%%%%%%%%%%%%%%%%%%%%%%%
\subsection{A probabilistic interpretation}
\label{sec:proba}
From a probabilistic point of view, the conclusion of
Theorem~\ref{thm:bus} means intuitively that if we choose two conductivities at
random, there is zero probability that $\Ps(\gamma_1,\gamma_2) \neq
\complex^E$. To see this consider an absolutely continuous (with
respect to the Lebesgue measure) random vector $(\Gamma_1,\Gamma_2)
\in \complex^E_0 \times \complex_0^E$ with distribution $\mu$ and
induced probability measure $\proba$. If the hypothesis of
theorem~\ref{thm:bus} hold and
$M \subset \complex^E_0 \times \complex_0^E$ is a measurable set with
$\proba\Mb{(\Gamma_1,\Gamma_2) \in M} > 0$ then
 \begin{equation}
  \proba\Mb{\Ps(\Gamma_1,\Gamma_2)= \complex^E ~|~ (\Gamma_1,\Gamma_2) \in M} = 1.
 \end{equation}
Indeed the conclusion of theorem~\ref{thm:bus} means that the set
\begin{equation}
Z = \Mcb{(\Gamma_1,\Gamma_2) \in M ~|~\Ps(\Gamma_1,\Gamma_2) \neq \complex^E }
\end{equation}
is a set of measure zero and therefore by absolute continuity of $\mu$
we also have $\mu(Z)=0$. Moreover the set $Z$ contains all the pairs
$(\gamma_1,\gamma_2)$ for which $\Lambda_{\gamma_1} =
\Lambda_{\gamma_2}$ but $\gamma_1 \neq \gamma_2$, i.e conductivities
that are indistinguishable from boundary measurements. Hence we also
have that the expectation
\begin{equation}
\expect\Mb{ \| \Lambda_{\Gamma_1,0} - \Lambda_{\Gamma_2,0} \| ~|~
(\Gamma_1,\Gamma_2) \in M} = \int_M \| \Lambda_{\Gamma_1,0} -
\Lambda_{\Gamma_2,0} \| ~d\mu > 0.
\end{equation}

The conclusion of Corollary~\ref{cor:sloc1} means that if $\Gamma$ is an
absolutely continuous random vector on $\complex_0^E$, then on any
measurable set $M \subset \complex_0^E$ with $\proba\Mb{\Gamma \in M}
>0$, we must have that 
\begin{equation}
 \proba\Mb{D_\gamma\Lambda_{\Gamma,0} ~\text{is injective} ~|~\Gamma \in
 M} = 1.
\end{equation}
We conclude this section by noting that uniqueness a.e. arises
naturally in (finite dimensional) linear systems.
\begin{example}
 Another example of uniqueness a.e. is that of a
 matrix that has a non-trivial nullspace. Let $A \in \complex^{m\times
 n}$, with $A \neq 0$ and $\nullspace(A) \neq \{0\}$. Then
 $\nullspace(A)$ is a subspace of $\complex^n$ of dimension at most
 $n-1$, and as such is a set of Lebesgue measure zero in $\complex^n$.
 Hence for almost all $(x_1,x_2) \in \complex^n \times \complex^n$, $A
 x_1 = A x_2$ implies $x_1=x_2$.
\end{example}
%%%%%%%%%%%%%%%%%%%%%%%%%%%%%%%%%%%%%%%%%%%%%%%%%%%%%%%%%%%%%%%%%%%%%%%%
\subsection{Proof of uniqueness a.e. for conductivities}
\label{sec:sbur}
First we note that the product of solutions subspace is spanned by
finitely many Dirichlet problem solutions.
\begin{lemma}\label{lem:fins}
Let $\gamma_1,\gamma_2 \in \complex^E_0$. Let $u^{(i)}$ (resp.
$v^{(i)}$) be $\gamma_1,0$ (resp. $\gamma_2,0$) harmonic with boundary
data $u^{(i)}_B = v^{(i)}_B = e_i$, where $\{e_i\}_{i \in B}$ is the
canonical basis of $\complex^B$. Then the product of solutions subspace
is such that
\begin{equation}
 \Ps(\gamma_1,\gamma_2) = \linspan \Mcb{(\DG u^{(i)}) \odot
 (\DG v^{(j)})}_{i,j \in B}.
\label{eq:fins}
\end{equation}
\end{lemma}
\begin{proof}
By definition of the product of solutions subspace, we must have 
\[ 
 \linspan \Mcb{ (\DG u^{(i)}) \odot (\DG v^{(j)})}_{i,j \in B} \subset
 \Ps(\gamma_1,\gamma_2).
\]
Let $w \in \Ps(\gamma_1,\gamma_2)$, then $w = (\DG u) \odot (\DG v)$,
where $u$ is $\gamma_1,0$ harmonic and $v$ is $\gamma_2,0$ harmonic.
Since the Dirichlet problems for $\gamma_1,0$ and $\gamma_2,0$ are
well-posed, we have 
\[
 u = \sum_{i\in B} u_B(i) u^{(i)} ~\text{and}~ 
 v = \sum_{j\in B} v_B(j) v^{(j)}.
\]
Hence $w$ can be written as a linear combination of the $(\DG u^{(i)}) \odot
(\DG v^{(j)})$ and we have the inclusion $\Ps(\gamma_1,\gamma_2) \subset
\linspan \Mcb{ (\DG u^{(i)}) \odot (\DG v^{(j)})}_{i,j \in B}$, which
gives the desired result.
\end{proof}

Another way to write the subspace $\Ps(\gamma_1,\gamma_2)$ is as the
range of a matrix, i.e.
\begin{equation}
 \Ps(\gamma_1,\gamma_2) = \colspace( \Ws ( \gamma_1,\gamma_2) ),
\end{equation}
where the matrix $\Ws(\gamma_1,\gamma_2) \in \complex^{|E| \times
|B|^2}$ is the matrix with columns being all possible Hadamard products
between the columns of the matrices $\nabla U(\gamma_1)$ and $\nabla U(\gamma_2)$,
where $U(\gamma)$ is the matrix with solutions corresponding to boundary
data $e_i$, $i\in B$, that is
\[
 U(\gamma) = 
 \begin{bmatrix}
 \mathbb{I}_B\\
 ((L_{\gamma})_{II})^{-1} L_{IB}
 \end{bmatrix}.
\]
where $\mathbb{I}_B$ is the $|B| \times |B|$ identity matrix. 
Concretely, the matrix $\Ws(\gamma_1,\gamma_2)$ is given by
\begin{equation}
 [\Ws(\gamma_1,\gamma_2)]_{:,i+(j-1)|B|} = \Mb{ \nabla U(\gamma_1) }_{:,i} \odot
 \Mb{ \nabla U(\gamma_2) }_{:,j}, ~\text{for}~i,j=1,\ldots,|B|,
 \label{eq:sprodmat}
\end{equation}
where the $i-$th column of a matrix $A$ is denoted by $A_{:,i}$. 

We are interested in finding a way of characterizing whether the
products of (gradients of) solutions span $\complex^E$ or not.  One way
to do this is to look at the determinants
\begin{equation}\label{eq:fs} 
\fs_\alpha(\gamma_1,\gamma_2) = \det [\Ws(\gamma_1,\gamma_2)]_{:,\alpha} 
\end{equation} 
where the matrix $[\Ws(\gamma_1,\gamma_2)]_{:,\alpha}$ is the $|E| \times |E|$
submatrix of $\Ws(\gamma_1,\gamma_2)$ obtained by selecting the $|E|$ columns
corresponding to the multi-index $\alpha \in \{1,\ldots,|B|^2\}^{|E|}$. Thus
$\Ps(\gamma_1,\gamma_2) = \complex^E$ if and only if there is an $\alpha \in
\{1,\ldots,|B|^2\}^{|E|}$ for which $\fs_{\alpha}(\gamma_1,\gamma_2) \neq 0$.
When $|B|^2<|E|$, all such determinants are zero (since we must have have
repeated columns). When $|B|^2 \geq |E|$, it is enough to check the
$\binom{|B|^2}{|E|}$ choices of columns $\alpha =
(\alpha_1,\ldots,\alpha_{|E|})$, with $1 \leq \alpha_1 < \ldots <
\alpha_{|E|} \leq |B|^2$. This observation is summarized in the
following lemma.

\begin{lemma}\label{lem:subs}
Let $\gamma_1,\gamma_2 \in \complex_0^E$. The following statements are equivalent
\begin{itemize}
 \item[i.] $\Ps(\gamma_1,\gamma_2) = \complex^E$
 \item[ii.] $\rank \Ws(\gamma_1,\gamma_2) = |E|$
 \item[iii.] $\fs_\alpha(\gamma_1,\gamma_2) \neq 0$ for some $\alpha \in
 \{1,\ldots,|B|^2\}^{|E|}$ with $1 \leq \alpha_1 < \ldots < \alpha_{|E|}
 \leq |B|^2$.
\end{itemize}
\end{lemma}

\begin{remark}
Finding the rank of $\Ws(\gamma_1,\gamma_2)$ by checking the
determinants of all possible square submatrices with maximal dimensions
is not very efficient and is notoriously inaccurate to calculate in
floating point arithmetic. A better way would be to test whether the
smallest singular value of $\Ws(\gamma_1,\gamma_2)$ is close to zero,
and this is what we have used in the numerics
(\S\ref{sec:numerics}). Nevertheless we use this determinantal
characterization of rank in the proof of theorem~\ref{thm:bus} because
it is an algebraic operation that preserves analyticity.
\end{remark}

The notion of analyticity that we use here is that of analytic (or
holomorphic) functions of several complex variables (see e.g.
\cite{Gunning:1965:AFS}). We recall that a function $f: \complex^n \to
\complex$ is analytic on some open set $U \subset \complex^n$ if for
each $z_0 \in U$, $f(z)$ can be expressed as a power series that
converges on $U$, i.e.
\[
 f(z) = \sum_{\alpha\in \nat^n} c_\alpha (z-z_0)^\alpha,
\]
where for a multi-index $\alpha = (\alpha_1,\ldots,\alpha_n) \in \nat^n$ we have for $z\in
\complex^n$ that $z^\alpha = z_1^{\alpha_1} \ldots z_n^{\alpha_n}$.
In particular, rational functions of the form $A(z)/B(z)$, for two
polynomials $A(z)$ and $B(z)$, are analytic on any connected open set where $B(z)
\neq 0$. Hence we have the following.

\begin{lemma}\label{lem:sfanalytic}
The functions $\fs_\alpha: \complex^E_0 \times \complex^E_0 \to \complex$ are
analytic for any $\alpha \in
\{0,\ldots,|B|^2\}^{|E|}$.
\end{lemma}
\begin{proof}
Let $\gamma_1,\gamma_2 \in \complex_0^E$.  First note that the entries of the
matrices $((L_{\gamma_k})_{II})^{-1}$, $k=1,2$, are complex analytic on
$\gamma_1$ and $\gamma_2$ when $\gamma_1,\gamma_2\in \complex^E_0$. This
is because the cofactor formula for the inverse guarantees that the
entries of the matrices in question are rational functions in
$\gamma_1,\gamma_2$ which are analytic provided $\det
((L_{\gamma_k})_{II}) \neq 0$, $k=1,2$.  As in the proof of
theorem~\ref{thm:dir}, the condition $\gamma_1,\gamma_2 \in
\complex^E_0$ ensures that $\det (L_{\gamma_k})_{II} \neq 0$, $k=1,2$.
Since $\Ws(\gamma_1,\gamma_2)$ is obtained from
$\DG((L_{\gamma_k})_{II})^{-1}(L_\gamma)_{IB}$, $k=1,2$ by taking
columnwise Hadamard products, each entry of $\Ws(\gamma_1,\gamma_2)$ is
also analytic on $\gamma_1$ and $\gamma_2$ when $\gamma_1,\gamma_2\in
\complex^E_0$. Finally taking the determinant of a matrix with analytic
entries is also analytic. 
\end{proof}

We are now ready to prove the main result.
\begin{proof}[Proof of Theorem~\ref{thm:bus}]
By lemma~\ref{lem:subs}, the theorem hypothesis means that there is
some multi-index $\beta \in \{1,\ldots,|B|^2\}^{|E|}$ such that
\[
 \fs_\beta(\rho_1,\rho_2) \neq 0, ~\text{for some $\rho_1,\rho_2 \in
 \complex^{|E|}_0$}.
\] 
Thus $\fs_\beta$ is not identically zero. In fact its zero set restricted
to $\complex^{|E|}_0 \times \complex^{|E|}_0$ 
\begin{equation}
Z(\fs_\beta) = \Mcb{ (\gamma_1,\gamma_2) \in \complex^{|E|}_0 \times
\complex^{|E|}_0 ~ | ~ \fs_\beta(\gamma_1,\gamma_2) = 0 },
\label{eq:zerosets}
\end{equation}
must be a set of measure zero (see e.g.
\cite{Gunning:1965:AFS}), where we use the Lebesgue measure on
$\complex^{|E|} \times \complex^{|E|}$. By lemma~\ref{lem:subs},
the subset $S$ of $\complex^{|E|}_0 \times \complex^{|E|}_0$ on which
$\Ps(\gamma_1,\gamma_2)
\neq \complex^E$ is
\[
 S = \bigcap_{\alpha \in \{1,\ldots,|B|^2\}^{|E|}} Z(\fs_\alpha).
\]
Since $Z(\fs_\beta)$ has measure zero, the set $S \subset Z(\fs_\beta)$ must
have measure zero as well by monotonicity of the Lebesgue measure.
\end{proof}

Using the notation in this section we get the following.
\begin{proof}[Proof of corollary~\ref{cor:newton}]
 We start by showing (i).  If the linearized problem at $\gamma$ is
 solvable, then there is some multi-index $\beta$ for which
 $\fs_\beta(\gamma,\gamma) \neq 0$.  The function $\fs^{(1)}:\complex
 \to \complex$ given by
 \[
  \fs^{(1)} (t) = \fs_\beta(\gamma + t\delta\gamma,\gamma +
 t\delta\gamma)
 \] 
 is an analytic function of the single variable $t$ on the set
 \[ 
  K = \{ t \in \complex ~|~ \Re(\gamma + t\delta\gamma) > 0 \}.
 \]
 The set $K$ is a finite intersection of open convex sets (open half
 planes) containing a neighborhood of the origin and is thus also open,
 convex and connected. Moreover, $\fs^{(1)}(t)$ is a rational function
 in $t$ with no poles in $K$ and thus it may have only finitely many
 $t\in K$ for which $\fs^{(1)}(t) = 0$.  Part (ii) follows similarly
 from considering the function $\fs^{(2)} : \complex \to \complex$ given
 by
 \[
  \fs^{(2)} (t) = \fs_\beta(\gamma ,\gamma +
 t\delta\gamma),
 \] 
 which is also a rational function in $t$ with no poles in $K$.
\end{proof}

%%%%%%%%%%%%%%%%%%%%%%%%%%%%%%%%%%%%%%%%%%%%%%%%%%%%%%%%%%%%%%%%%%%%%%%%
\section{Solvability for the inverse Schr\"odinger problem}
\label{sec:schroe}
The same technique can be used to study the solvability for the inverse
Schr\"odinger problem. First we relate a difference in boundary data to a
difference of Schr\"odinger potentials  (\S\ref{sec:qint}). The
appropriate products of solutions subspace is defined in
\S\ref{sec:qprod}.  The proof of the uniqueness result 
is deferred to Appendix~\ref{sec:qbur}.

%%%%%%%%%%%%%%%%%%%%%%%%%%%%%%%%%%%%%%%%%%%%%%%%%%%%%%%%%%%%%%%%%%%%%%%%
\subsection{An interior identity}
\label{sec:qint}

Let us write a relation similar to the continuum relation in
\cite{Sylvester:1987:GUT} that relates a difference in boundary data to
a difference in Schr\"odinger potentials.

\begin{lemma}[Interior identity for Schr\"odinger potentials]
\label{lem:qint} 
Let $\gamma \in \complex_0^E$ and $q_1,q_2 \in \complex^I$ be such that
the $\gamma,q_1$ and $\gamma,q_2$ Dirichlet problems are well-posed. Let
$u$ be $\gamma,q_1$ harmonic and $v$ be $\gamma,q_2$ harmonic. Then we
have the following identity
\begin{equation}
\label{eq:qint}
u_B^T ( \Lambda_{\gamma,q_1} - \Lambda_{\gamma,q_2}) v_B = u_I^T
\diag(q_1 - q_2) v_I = (q_1 - q_2)^T(u_I \odot v_I).
\end{equation}
\end{lemma}
\begin{proof}
The proof is very similar to the proof in the continuum 
\cite{Sylvester:1987:GUT}. Since $u$ is $\gamma,q_1$ harmonic,
lemma~\ref{lem:gi} guarantees that
\begin{equation}
 u_B^T \Lambda_{\gamma,q_1} v_B = u^T L_\gamma v + u_I^T \diag(q_1) v_I.
 \label{eq:gi1}
\end{equation}
 Since $v$ is $\gamma,q_2$
harmonic we also have
\begin{equation}
 u_B^T \Lambda_{\gamma,q_2} v_B = u^T L_\gamma v + u_I^T \diag(q_2) v_I.
 \label{eq:gi2}
\end{equation}
The desired result is obtained by subtracting \eqref{eq:gi2} from
\eqref{eq:gi1}.
\end{proof}

%%%%%%%%%%%%%%%%%%%%%%%%%%%%%%%%%%%%%%%%%%%%%%%%%%%%%%%%%%%%%%%%%%%%%%%%
\subsection{Uniqueness almost everywhere}
\label{sec:qprod}
As in the uniqueness proof for the continuum Schr\"odinger
problem \cite{Sylvester:1987:GUT}, we would
like to study the subspaces of $\complex^I$,
\begin{equation}
\Pq(q_1,q_2) \equiv
 \linspan \Mcb{ u_I \odot v_I ~|~ \text{$u$ is $\gamma,q_1$ harmonic and
 $v$ is $\gamma,q_2$ harmonic} },
\label{eq:pqq}
\end{equation}
for potentials $q_1$ and $q_2$ that ensure the corresponding Dirichlet
problems are well-posed.  If the two potentials $q_1$ and $q_2$ gave
identical boundary data, i.e. $\Lambda_{\gamma,q_1} =
\Lambda_{\gamma,q_2}$, then lemma~\ref{lem:qint} implies that
$\overline{q}_1 - \overline{q}_2 \in \Pq(q_1,q_2)^\perp$. If in
addition we had that $\Pq(q_1,q_2) = \complex^I$ we could conclude that
$q_1 = q_2$, which gives uniqueness. The following theorem is very
similar to the uniqueness theorem~\ref{thm:bus} for conductivities, but
with Dirichlet well-posedness conditions that are slightly more
complicated. 
\begin{theorem}[Uniqueness almost everywhere for Schr\"odinger potentials]
\label{thm:bur} 
Let $\gamma \in \complex_0^E$ and $\zeta = -\lambda_{\min}(\Re
(L_\gamma)_{II})$.  If $\Pq(p_1,p_2) = \complex^I$ for some $p_1,p_2 \in
\complex^I_\zeta$, then $\Pq(q_1,q_2) = \complex^I$ for almost all
$(q_1,q_2) \in \complex^I_\zeta \times \complex^I_\zeta$.
\end{theorem}
Since the proof of theorem~\ref{thm:bur} is very similar to that of
theorem~\ref{thm:bus}, it is deferred to Appendix~\ref{sec:qbur}.
Clearly, if the hypothesis of the theorem holds, $\Lambda_{\gamma,q_1} =
\Lambda_{\gamma,q_2}$ implies $q_1 = q_2$, for all $(q_1,q_2) \in
\complex^I_\zeta \times \complex^I_\zeta$, except for a set of Lebesgue
measure zero in $\complex^I_\zeta \times \complex^I_\zeta$.  The
linearization of the boundary data $\Lambda_{\gamma,q}$ with respect to
changes in the potential $q$ is as follows.

\begin{lemma}[Linearization for Schr\"odinger potentials]\label{lem:linq}
Assume the Dirichlet problem for $\gamma,q$ is well-posed, then for
sufficiently small $\delta q \in \complex^I$, and any $u_B,v_B \in
\complex^B$,
\begin{equation}
 u_B^T \Lambda_{\gamma,q + \delta q} v_B = u_B^T \Lambda_{\gamma,q} v_B +
 (\delta q)^T (u_I \odot v_I) + o(\delta q),
\end{equation}
where $u=(u_B,u_I)$ and $v = (v_B,v_I)$ are $\gamma,q$ harmonic.
\end{lemma}
\begin{proof}
Use the interior identity \eqref{eq:qint} with $q_1 = q+\epsilon \delta q$
and $q_2 = q$, divide by $\epsilon$ and take the limit as $\epsilon \to
0$.
\end{proof}

The previous lemma shows that the mapping $q \to \Lambda_{\gamma,q}$ is
Fr\'echet differentiable. The {\em Jacobian} $D_q\Lambda_{\gamma,q}:
\complex^I \to \complex^{B \times B}$ of $\Lambda_{\gamma,q}$ about $q$
is
\begin{equation}
 D_q\Lambda_{\gamma,q} \delta q = L_{BI} (L_{II} + \diag(q))^{-1}
 \diag(\delta q)  (L_{II} + \diag(q))^{-1} L_{IB},
\end{equation}
where $\delta q \in \complex^I$ and for clarity we omitted the subscript
$\gamma$ in the weighted graph Laplacian $L_\gamma$. We say the
linearized problem is {\em solvable} at $q$, when the Jacobian about $q$
is injective, i.e. $\nullspace(D_q\Lambda_{\gamma,q}) = \{0\}$.
Solvability of the linearized problem at 
$q$ is of course equivalent to $\colspace(D_q\Lambda_{\gamma,q}^*) =
\complex^I = \Pq(q,q)$, where the last equality comes from
lemma~\ref{lem:linq} and the definition of the Jacobian. Hence we have
the following corollaries, which are stated without proof because of
their similarity to those for the conductivity problem.

\begin{corollary}\label{cor:qloc0}
 Let $\gamma \in \complex_0^E$ and $\zeta$ be such that the $\gamma,q$
 Dirichlet problem is well posed when $q \in \complex_\zeta^I$. If there
 is $p_1,p_2 \in \complex_\zeta^I$ such that
 $\Pq(p_1,p_2) = \complex^I$, then any equivalence class of
 Schr\"odinger potentials in $\complex_\zeta^I$ for the equivalence relation
 $\Lambda_{\gamma,q_1}=\Lambda_{\gamma,q_2}$ must be of measure zero in
 $\complex_\zeta^I$.
\end{corollary}

\begin{corollary}\label{cor:qloc2}
 Let $\gamma \in \complex_0^E$ and $\zeta$ be such that the $\gamma,q$
 Dirichlet problem is well posed when $q \in \complex_\zeta^I$.
 If the linearized problem is solvable at some $p \in \complex^I_\zeta$,
 then $\Lambda_{\gamma,q_1} = \Lambda_{\gamma,q_2}$ implies $q_1 = q_2$,
 for almost all $(q_1,q_2) \in \complex^I_\zeta \times \complex^I_\zeta$.
\end{corollary}

\begin{corollary}\label{cor:qloc1}
 Let $\gamma \in \complex_0^E$ and $\zeta$ be such that the $\gamma,q$
 Dirichlet problem is well posed when $q \in \complex_\zeta^I$.
 If the linearized problem is solvable about some $p \in
 \complex^I_\zeta$, the
 linearized problem is solvable for almost all $q \in \complex^I_\zeta$.
\end{corollary}

\begin{corollary}\label{cor:qnewton}
Let $\gamma \in \complex_0^E$ and $\zeta$ be such that the $\gamma,q$
 Dirichlet problem is well posed when $q \in \complex_\zeta^I$.
If the linearized problem is solvable at $q \in \complex^I_\zeta$ then
along any direction $\delta q \in \complex^I$ there are at most finitely
many $t\in \complex$ with $\Re(q + t \delta q) > 0$ for which
\begin{itemize}
 \item[i.] the linearized problem
 is not solvable at $q + t \delta q$.
 \item[ii.]  $q$ and $q+t\delta q$ may have the same DtN
 map (i.e. $\Pq(q,q + t \delta q) \neq \complex^I$).
\end{itemize}
\end{corollary}

As in the case of conductivities, corollary~\ref{cor:qnewton} guarantees
that when Newton's algorithm is used to find the Schr\"odinger potential
in a graph, it is very likely that the Newton systems that are solved at
each iteration have a unique solution (except possibly at finitely many
points). Corollary~\ref{cor:qloc2} suggests that checking whether the
linearized problem about $p=0$ is solvable is enough to guarantee
Schr\"odinger potentials in $\complex^I_\zeta$ can be recovered (up to a
zero measure set). Of course this is only a sufficient condition, and
this test may miss some rare cases where $\Pq(0,0) \neq \complex^I$ and
yet the Schr\"odinger problem is almost everywhere solvable. Finally
probabilistic interpretations similar to those in \S\ref{sec:proba} can
be made for theorem~\ref{thm:bur} and corollaries~\ref{cor:qloc1} and
\ref{cor:qloc2}.
\begin{remark}\label{rem:qreal}
As in remark~\ref{rem:sreal}, we can get results analogous to
theorem~\ref{thm:bur} and corollaries~\ref{cor:qloc0}, \ref{cor:qloc1},
\ref{cor:qloc2} and \ref{cor:qnewton}, but where $\complex$ is replaced
by $\real$. The set $\real_\zeta^I$ is to be understood as the set of $q
\in \real^I$ for which $q>\zeta$, where the inequality is componentwise.
\end{remark}

%%%%%%%%%%%%%%%%%%%%%%%%%%%%%%%%%%%%%%%%%%%%%%%%%%%%%%%%%%%%%%%%%%%%%%%%
\section{Numerical study}
\label{sec:numerics}
We start by explaining in \S\ref{sec:tests} how to use the singular
value decomposition and our theoretical results to check solvability on
a graph. Examples of non-planar graphs where either the conductivity or
Schr\"odinger problems are solvable are given in \S\ref{sec:examples}.
Our a.e. uniqueness results allow for conductivities (or Schr\"odinger
potentials) in zero measure sets to have the same boundary data. We
visualize slices of these zero measure sets in \S\ref{sec:zero}. Finally
we present in \S\ref{sec:stat} a statistical study of recoverability,
which shows that it is much more likely for the Schr\"odinger problem to
be solvable on a graph than the conductivity problem. All the code
needed for reproducing the figures and numerical results in this section
is available in \cite{thecode}.

%%%%%%%%%%%%%%%%%%%%%%%%%%%%%%%%%%%%%%%%%%%%%%%%%%%%%%%%%%%%%%%%%%%%%%%%
\subsection{Recoverability tests}
\label{sec:tests}
Corollaries \ref{cor:sloc2} and \ref{cor:qloc2} give us a relatively
simple test to check  whether the conductivity (or Schr\"odinger)
problem on a graph is recoverable (in the weak sense we consider). All
we need to do is estimate the rank of the matrices $D_\gamma
\Lambda_{1,0}$ (for the conductivity problem) or $D_q \Lambda_{1,0}$
(for the Schr\"odinger problem). If the rank of these matrices is at
least the number of degrees of freedom ($|E|$ for the conductivity and
$|I|$ for the Schr\"odinger problem), then the problem is recoverable.
Here we chose for simplicity to linearize about the constant
conductivity of all ones or about the zero Schr\"odinger potential.
This choice of linearization point is rather arbitrary and could give
false negatives, i.e.  graphs on which the problem is recoverable but
where the test says otherwise. One could remedy this by choosing the
linearization point at random, but we did not see significant changes
our numerical experiments because of this.

For the mathematical argument we used determinants to find the rank of a
matrix. In our numerical experiments, we use instead the singular value
decomposition, as it is numerically robust. Following e.g.
\cite{Golub:2013:MC}, we say that a $n\times m$ matrix $A \neq 0$ with
$n\leq m$  has rank $r$ at a tolerance $\delta>0$ if we have
\begin{equation}
 1 \geq \frac{\sigma_2}{\sigma_1} \geq \ldots \geq
 \frac{\sigma_r}{\sigma_1} > \delta
 \geq \frac{\sigma_{r+1}}{\sigma_1} \geq \ldots \geq
 \frac{\sigma_n}{\sigma_1},
\end{equation}
where $\sigma_1 \geq \ldots \geq \sigma_n$, are the singular values of $A$.

%%%%%%%%%%%%%%%%%%%%%%%%%%%%%%%%%%%%%%%%%%%%%%%%%%%%%%%%%%%%%%%%%%%%%%%%
\subsection{Some recoverable graphs}
\label{sec:examples}
We give some examples of graphs where the conductivity
(figure~\ref{fig:ssample}) and the Schr\"odinger potential
(figure~\ref{fig:qsample}) are recoverable.  The examples of
figures~\ref{fig:ssample} and \ref{fig:qsample} illustrate that our theory
allows for graphs that are not planar.

\begin{figure}
 \begin{center}
 \begin{tabular}{c@{\hspace{3em}}c}
  \includegraphics[width=0.30\textwidth]{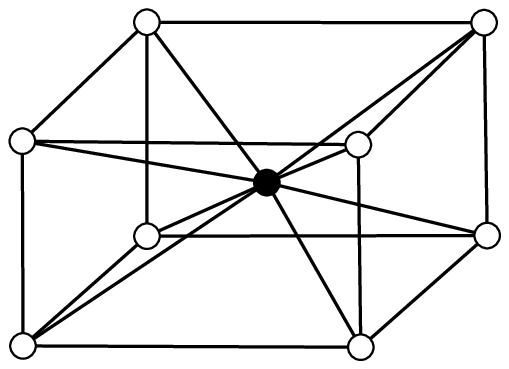}&
  \includegraphics[width=0.30\textwidth]{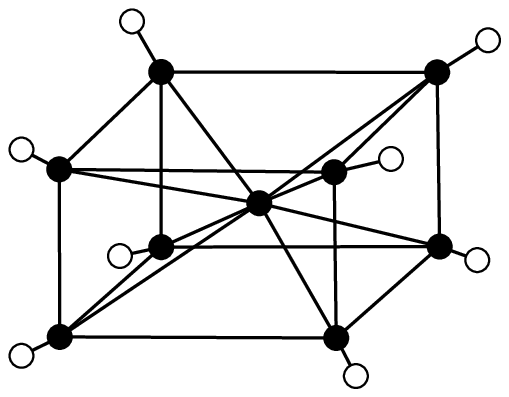}\\
  (a) & (b)
 \end{tabular}
 \end{center}
 \caption{Examples of non-planar graphs where the conductivity is
 recoverable. The boundary nodes are in white and the interior
 nodes in black.}
 \label{fig:ssample}
\end{figure}

\begin{figure}
 \begin{center}
 \begin{tabular}{c@{\hspace{3em}}c}
  \includegraphics[width=0.30\textwidth]{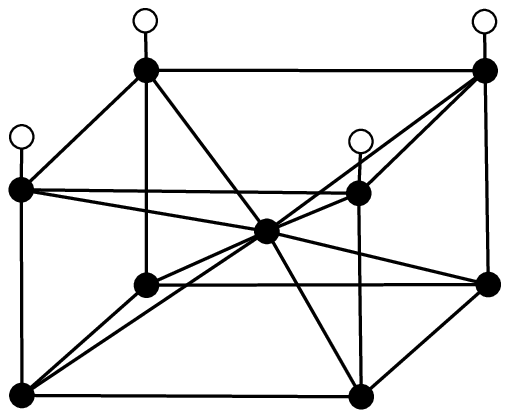} &
  \includegraphics[width=0.30\textwidth]{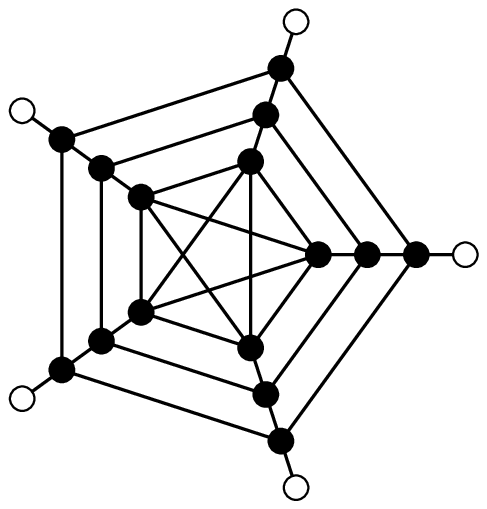}\\
  (a) & (b)
 \end{tabular}
 \end{center}
 \caption{Examples of non-planar graphs where the Schr\"odinger
 potential is recoverable. The boundary nodes are in white and the
 interior nodes in black.}
 \label{fig:qsample}
\end{figure}

%%%%%%%%%%%%%%%%%%%%%%%%%%%%%%%%%%%%%%%%%%%%%%%%%%%%%%%%%%%%%%%%%%%%%%%%
\subsection{Examples of indistinguishable conductivities and
Schr\"odinger potentials}
\label{sec:zero}
We can also get a glimpse of the conductivities (or Schr\"odinger
potentials) that could have the same DtN maps. For the conductivity
problem we fix the topology of the graph, choose two directions
$\delta\gamma_1$ and $\delta\gamma_2$, and compute the smallest singular
value (rescaled by the largest)  of the products of gradients matrix
$\Ws(\gamma_1(x),\gamma_2(y))$, with $\gamma_1(x) = 1 + x
\delta\gamma_1$ and $\gamma_2(y) = 1 + y \delta\gamma_2$, for many
$(x,y) \in \real_0^2$ (the products of gradients matrix is defined in
\eqref{eq:sprodmat}). An example of this
is shown if figure~\ref{fig:szero}. Knowing
$\sigma_{\min}/\sigma_{\max}$ (the reciprocal of the conditioning
number) for $\Ws(\gamma_1,\gamma_2)$ is useful when measurements are
tainted by noise. Indeed if the norm of the noise relative to norm of
the measurements is larger than $\sigma_{\min}/\sigma_{\max}$, then for
all practical purposes we may not be able to distinguish these two
conductivities $\gamma_1$ and $\gamma_2$. Thus when there is noise in
the data, we may get a set of conductivities that is not of measure
zero, but that are indistinguishable from boundary data. This is related
to the concept of indistinguishable perturbations in the continuum
conductivity problem \cite{Gisser:1990:ECC}.

We proceed similarly for the Schr\"odinger problem. We fix $\gamma=1$,
choose two Schr\"odinger potentials $q_1$ and $q_2$ and compute the
smallest singular value (rescaled by the largest) of the product of
solutions matrix $\Wq(x q_1,y q_2)$ (to be defined in
\eqref{eq:qprodmat}), for many $(x,y) \in \real_0^2$. This is
illustrated in figure~\ref{fig:qzero}. The same observation as for the
conductivity applies: if the ratio of the norm of the noise to the
norm of the measurements is less than $\sigma_{\min}/\sigma_{\max}$ for
$\Wq(q_1,q_2)$, then for all practical purposes we may not be able to
distinguish between $q_1$ and $q_2$ from boundary measurements.
\begin{figure}
 \begin{center}
 \begin{tabular}{c@{\hspace{3em}}c}
 \raisebox{2em}{\includegraphics[width=0.3\textwidth]{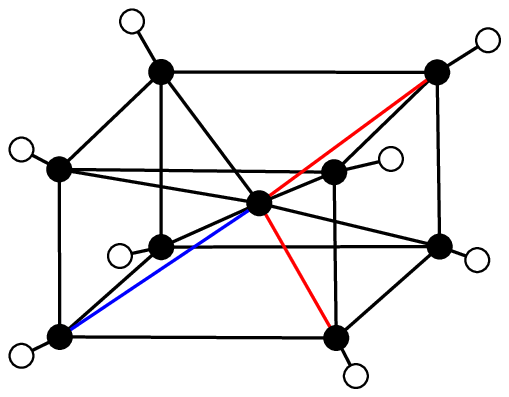}}& 
 \includegraphics[width=0.4\textwidth]{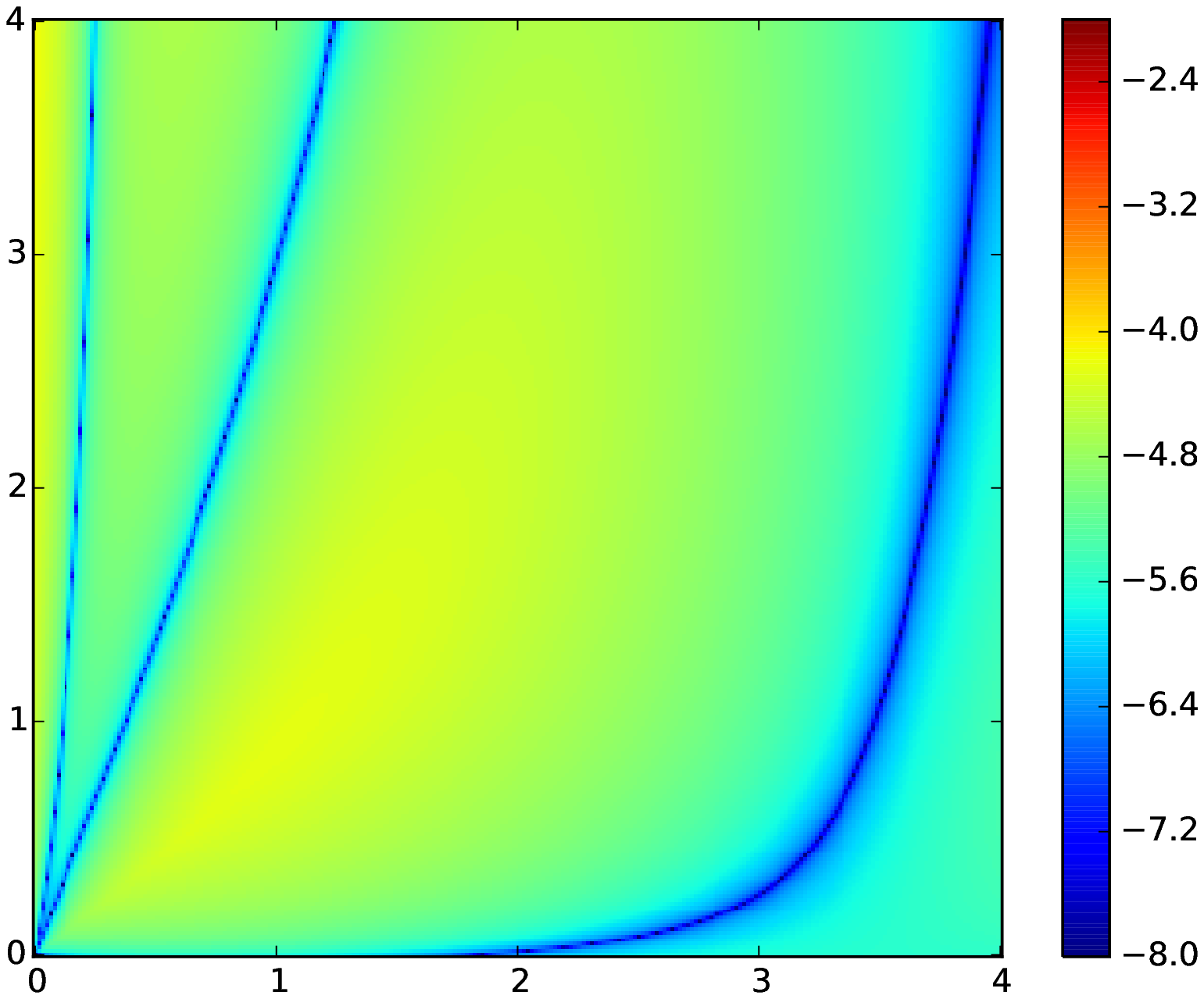}
 \end{tabular}
 \end{center}
 \caption{The smallest singular value (rescaled by the largest) of the
 product of solutions matrices
 $\Ws(1+x\delta\gamma_1,1+y\delta\gamma_2)$ for the graph in
 figure~\ref{fig:ssample}(b) and $(x,y) \in [0,4]^2$.  The
 direction $\delta \gamma_1$ (resp. $\delta \gamma_2$) is the characteristic
 function of the red (resp. blue) edges. The color scale on the right
 corresponds to $\log_{10}(\sigma_{\min}/\sigma_{\max})$.
 }\label{fig:szero}
\end{figure}

\begin{figure}
 \begin{center}
 \begin{tabular}{c@{\hspace{3em}}c}
 \raisebox{0em}{\includegraphics[width=0.3\textwidth]{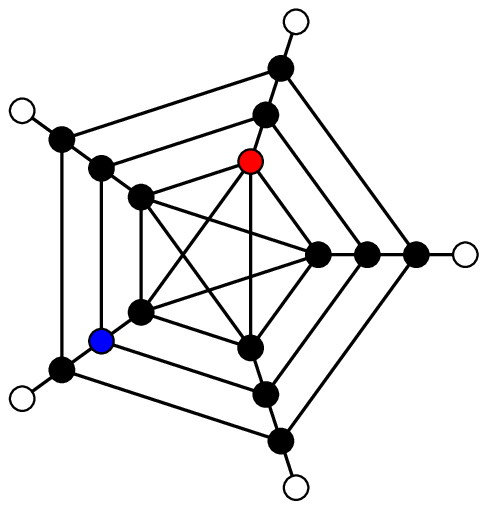}}& 
 \includegraphics[width=0.4\textwidth]{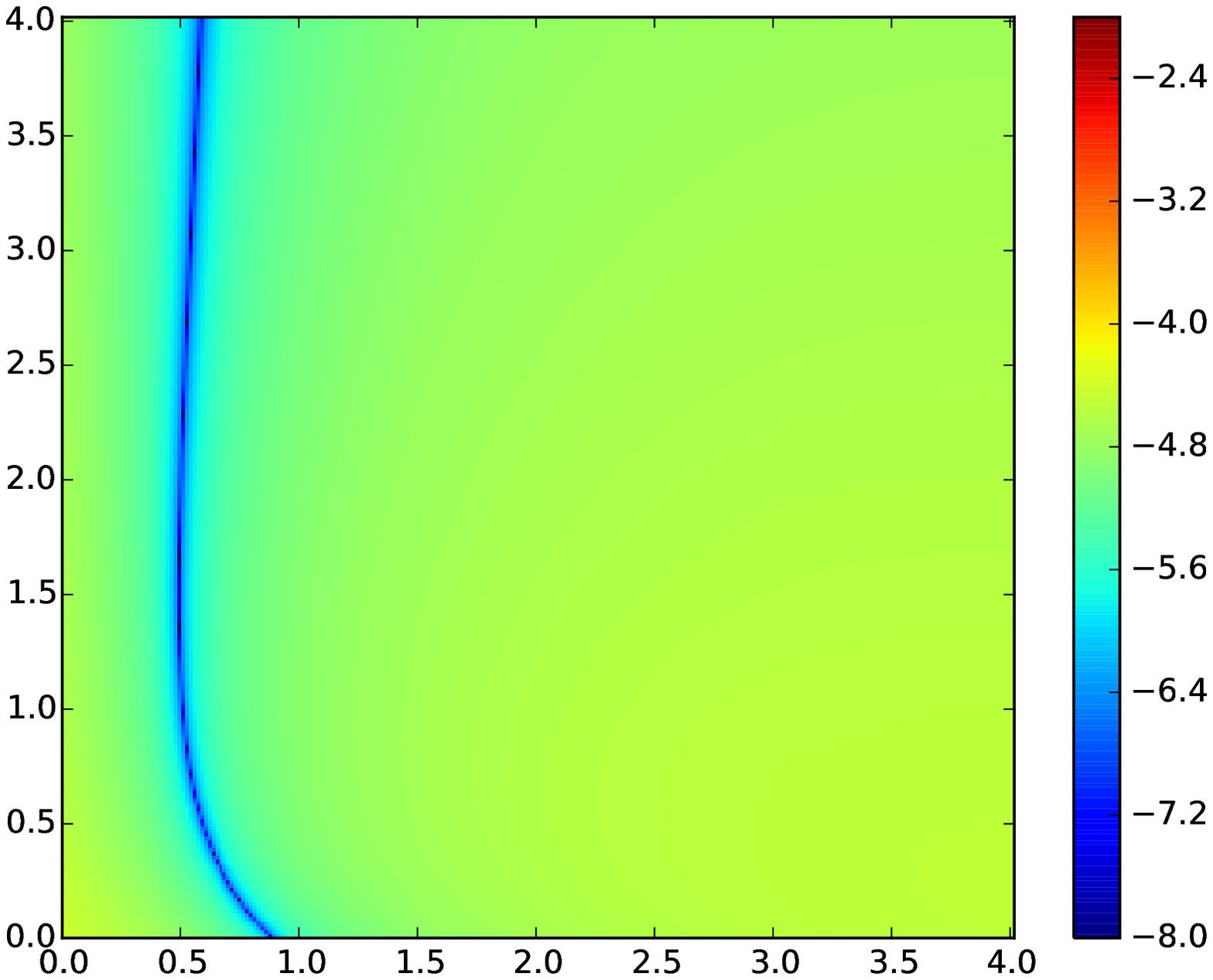}
 \end{tabular}
 \end{center}
 \caption{The smallest singular value (rescaled by the largest ) of the
 product of solutions matrices $\Wq(xq_1,yq_2)$ for the graph in
 figure~\ref{fig:qsample}(b) and $(x,y) \in [0,4]^2$. The Schr\"odinger
 potential $q_1$ (resp. $q_2$) is the characteristic function of the red
 (resp. blue) node. The color scale on the right corresponds to
 $\log_{10}(\sigma_{\min}/\sigma_{\max})$.  }\label{fig:qzero}
 \end{figure}

%%%%%%%%%%%%%%%%%%%%%%%%%%%%%%%%%%%%%%%%%%%%%%%%%%%%%%%%%%%%%%%%%%%%%%%%
\subsection{Statistical study of recoverability}
\label{sec:stat}
A natural question to ask is how likely is it for a graph to be
recoverable? To answer this question we performed the following
statistical studies. 

\subsubsection{Conductivity problem} We used one of the
Erd\H{o}s-R\'enyi random graph models \cite{Erdos:1959:ORG} to generate
graphs that have a fixed number of edges $|E|$ and vertices. In this
model $|E|$ edges are drawn (uniformly) from all possible $n(n-1)/2$
edges on a graph with $n$ vertices.   We choose this particular random
graph model because we would like to compare problems with the same
number of degrees of freedom. Concretely, the statistical study involved
varying two parameters (while keeping $|E|$ fixed): the first is the
number of internal nodes $|I|$ and the second the number of boundary
nodes $|B|$. To ensure the Dirichlet problem is well posed we draw only
graphs that are connected and whose restriction to the interior nodes is
connected as well.  For each combination of $|I|$ and $|B|$, we draw up
to $200$ random graphs connected as a whole and restricted to $I$. In
each draw we make up to $20$ trials to make sure the Dirchlet problem is
well-posed. The image shown in figure~\ref{fig:stat}(a) represents the
(empirical) probability of recoverability for a {\em connected} graph
(as a whole and restricted to the interior nodes) with $|E|=21$, and
$|B|$, $|I|$ given.

There is one phase transition as we vary $|B|$ that can be explained by
a counting argument. Indeed the DtN map $\Lambda_{\gamma,0}$ is a
symmetric matrix with zero row sums, and is thus determined by
$|B|(|B|-1)/2$ scalars. So if $|B|(|B|-1)/2 < |E|$ we do not have enough
data to recover all edges in the graph. We also see that if there are no
interior nodes, then the problem can be solved almost always. However as
we increase the number of internal nodes, the graphs on which the
conductivity problem is recoverable become scarce. This may be because
the random graph model we chose is not adapted to study this particular
recoverability problem. We point out that the critical circular
planar graphs \cite{Curtis:1998:CPG} have much more internal nodes, so
recoverability for conductivities seems to be a very rare property (for
$|B| = 7$, a critical planar graph called ``pyramidal network''
\cite{Curtis:1998:CPG} has $8$ interior nodes).

\subsubsection{Schr\"odinger problem} Here we use the other
Erd\H{o}s-R\'enyi random graph model \cite{Erdos:1959:ORG} to generate
graphs with a fixed number of vertices, but where the probability of
having an edge between two vertices is $p$. The statistical study
consisted of keeping $|I|$ fixed (i.e. the number of degrees of freedom
in the Schr\"odinger potential $q$) and changing two parameters: one is
the probability $p$ of having an edge between any two nodes and the
other the number of boundary nodes $|B|$. The results are displayed in
figure~\ref{fig:stat}(b). As for the conductivity problem we display
(out of a maximum $200$ trials) the (empirical) probability of
recoverability for a {\em connected} graph (as a whole and restricted to
the interior nodes) with $|I|=21$ and $|B|$, $p$ given.

We observe two phase transition like properties. The first one follows
from a counting argument: if there is not enough data to recover the
unknowns then the problem is not recoverable. To be more precise, the
DtN map $\Lambda_{\gamma,q}$ is a symmetric $|B| \times |B|$ matrix that
is determined by $|B|(|B|+1)/2$ scalars. So if $|B|(|B|+1)/2 < |I|$,
then the problem is not recoverable. The second is that for
probabilities in what appears to be a symmetric interval about $1/2$,
the Schr\"odinger problem is almost always recoverable, whereas
elsewhere it is almost always not recoverable. We only have a heuristic
explanation for this: if the graph has too little edges, it may not be
possible to access certain internal nodes reliably. Also if the graph
has too many edges (say it is the complete graph), then it becomes
impossible to distinguish the current that flows through the leak from
the currents between an internal node an its neighbors.  Finally, the
interval of edge probabilities for which the Schr\"odinger problem is
recoverable becomes larger as we increase $|B|$. This confirms the
intuition that the larger $|B|$ is, the more data we have and the more
likely it is for us to find a recoverable graph.

\begin{figure}
 \begin{center}
 \begin{tabular}{cc}
 \raisebox{6em}{\rotatebox{90}{$|B|$}}
 \includegraphics[width=0.41\textwidth]{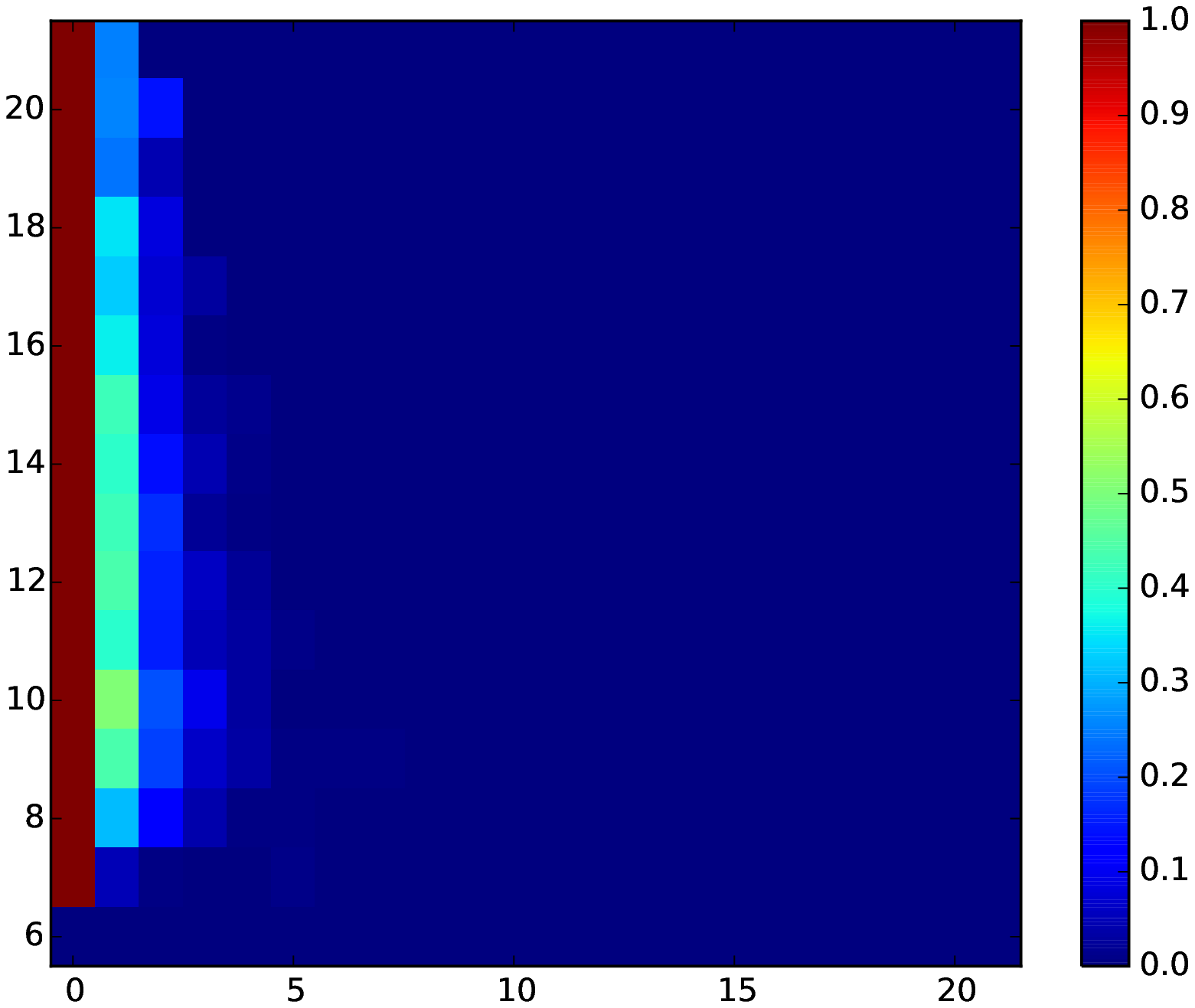} & 
 \raisebox{6em}{\rotatebox{90}{$|B|$}}
 \includegraphics[width=0.41\textwidth]{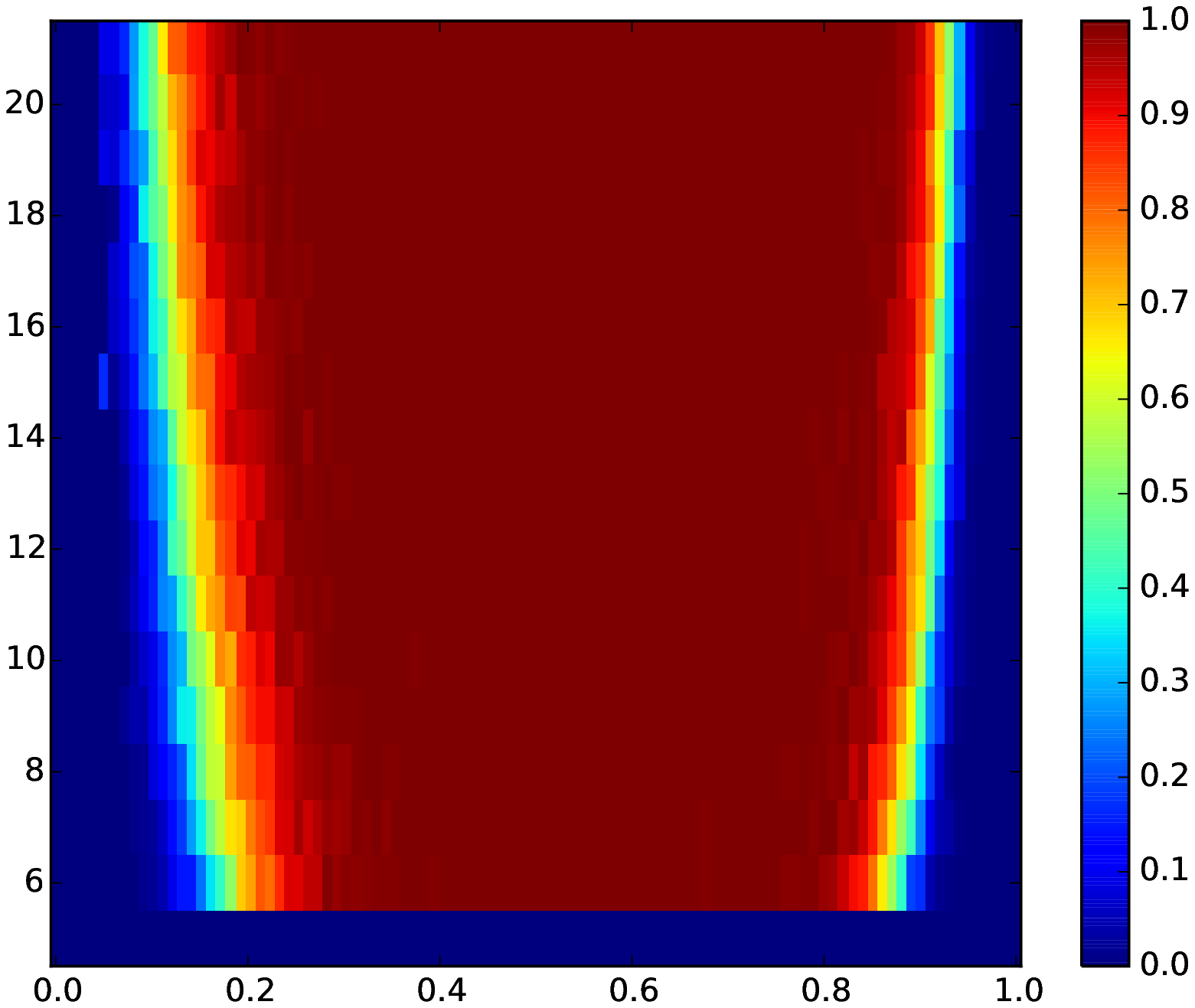} \\[-0.8em]
 $|I|$ & $p$\\
 (a) Conductivity with $|E|=21$ & (b) Schr\"odinger with $|I|=21$
 \end{tabular}
 \end{center}
 \caption{Statistical study of recoverability.}
 \label{fig:stat}
\end{figure}

%%%%%%%%%%%%%%%%%%%%%%%%%%%%%%%%%%%%%%%%%%%%%%%%%%%%%%%%%%%%%%%%%%%%%%%%
\section{Discussion and future work}
\label{sec:discussion}
Our results generalize existing solvability results for the conductivity
and Schr\"odinger problems on graphs
\cite{Curtis:1990:DRN,Curtis:1994:FCC,Colin:1994:REP,Curtis:1998:CPG,Chung:2005:HFI,Chung:2010:IRE,Lam:2012:IPC,Arauz:2014:DRM,Arauz:2015:OPB,
Arauz:2015:DRM} by relaxing what is meant by solvability and allowing
conductivities (or Schr\"odinger potentials) in zero measure sets to
have the same boundary data. Thus the problem is ill-posed. If errors
are present in the measurements, the ill-posedness is worse since these
zero measure sets ``expand'' to positive measure sets. This is to be
expected as even when the discrete conductivity problem is known to be
solvable (e.g. the circular planar graph case) the problem become
increasingly ill-posed with the size of the network. Studying
regularization techniques to deal with this ill-posedness is left for
future work.

The complex geometric approach
\cite{Sylvester:1987:GUT,Calderon:1980:IBV} has been successfully used
to prove uniqueness for various continuum inverse problems
\cite{Uhlmann:1999:DIP}. We believe that the approach presented here can
also be extended to other discrete inverse problems, such as the inverse
problem of finding the spring constants, masses and dampers in a network
of springs, masses and dampers from displacement and force measurements
at a few nodes that are accessible.

Finally we would like to use the theoretical results presented here to
generalize the numerical methods in \cite{Borcea:2012:RNA} to deal with
the electrical impedance tomography problem and other related inverse
problems in setups that are not necessarily 2D and that involve complex
valued quantities. The observation in the numerics (\S7) that graphs on
which the Schr\"odinger problem is solvable are much more common than
graphs on which the conductivity problem is solvable indicates that to
generalize the methods in \cite{Borcea:2012:RNA} it may be beneficial to
first transform the inverse problem at hand to Schr\"odinger form (via
e.g. the Liouville identity \cite{Sylvester:1987:GUT}), image
a Schr\"odinger potential and then go back to the original quantity of
interest (e.g. the conductivity).

\section*{Acknowledgements}
This work was partially supported by the National Science Foundation
grant DMS-1411577. FGV is grateful to Alexander Mamonov for inspiring
conversations on network inverse problems during the MSRI special
semester on inverse problems (Fall 2010).  

\appendix
%%%%%%%%%%%%%%%%%%%%%%%%%%%%%%%%%%%%%%%%%%%%%%%%%%%%%%%%%%%%%%%%%%%%%%%%
\section{Proof of uniqueness a.e. for Schr\"odinger
potentials}
\label{sec:qbur}
We now proceed with the proof of the uniqueness a.e. result
(theorem~\ref{thm:bur}). The proof is essentially the same as that for
the conductivity case in \S\ref{sec:sbur}, and is included here for
completeness. We start by noticing that the product of solutions
subspace is spanned by a finite number of vectors.  This is the
objective of the following lemma, which is stated without proof because
of its similarity to lemma~\ref{lem:fins}.
\begin{lemma}\label{lem:finq}
 Assume the $\gamma,q_1$ and $\gamma,q_2$ Dirichlet problems are
 well-posed.  Let $u^{(i)}$ (resp. $v^{(i)}$) be $\gamma,q_1$ (resp.
 $\gamma,q_2$) harmonic with boundary data $u^{(i)}_B = v^{(i)}_B =
 e_i$, where $\{e_i\}_{i\in B}$ is the canonical basis of $\complex^B$.
 Then the product of solutions subspace $\Pq(q_1,q_2)$ is
\begin{equation}
 \Pq(q_1,q_2) = \linspan \Mcb{ u_I^{(i)} \odot v_I^{(j)}}_{i,j \in B}.
 \label{eq:finq}
\end{equation}
\end{lemma}
The product of solutions space \eqref{eq:finq} can be rewritten as the
range of a matrix,
\begin{equation}
 \Pq(q_1,q_2) = \colspace ( \Wq(q_1,q_2) ),
\end{equation}
where the matrix $\Wq(q_1,q_2) \in \complex^{|I| \times |B|^2}$ is the
matrix with columns being all the possible Hadamard products of the
columns of $(L_{II} + \diag(q_1))^{-1} L_{IB}$ and $(L_{II} +
\diag(q_2))^{-1} L_{IB}$, i.e.
\begin{equation}\label{eq:qprodmat}
[\Wq(q_1,q_2)]_{:,i+(j-1)|B|} = \Mb{(L_{II} + \diag(q_1))^{-1} L_{IB} }_{:,i} \odot
\Mb{(L_{II} + \diag(q_2))^{-1} L_{IB} }_{:,j},
\end{equation}
and the subscript $\gamma$ in the weighted graph Laplacian $L_\gamma$
has been omitted for clarity. 
As in the conductivity case, we want to characterize whether the
products of solutions span $\complex^I$ or not. One way to do this is to
look at the determinants
\begin{equation}\label{eq:f} 
\fq_\alpha(q_1,q_2) = \det [\Wq(q_1,q_2)]_{:,\alpha} 
\end{equation} 
where the matrix
$[\Wq(q_1,q_2)]_{:,\alpha}$ is the $|I| \times |I|$ submatrix of
$\Wq(q_1,q_2)$ obtained by selecting the $|I|$ columns corresponding to
the multi-index $\alpha \in \{1,\ldots,|B|^2\}^{|I|}$. Clearly
$\Pq(q_1,q_2) = \complex^I$ if and only if there is an $\alpha \in
\{1,\ldots,|B|^2\}^{|I|}$ for which $\fq_{\alpha}(q_1,q_2) \neq 0$. If
$|B|^2<|I|$, then we always have repeated columns and the determinants
are zero.  When $|B|^2 \geq |I|$, the determinant properties guarantee
that it is enough to check the $\binom{|B|^2}{|I|}$ choices of columns
$\alpha = (\alpha_1,\ldots,\alpha_{|I|})$, with $1 \leq \alpha_1 <
\ldots < \alpha_{|I|} \leq |B|^2$. This observation is summarized in the
following lemma.

\begin{lemma}\label{lem:subq}
Assume the Dirichlet problems for $\gamma,q_1$ and $\gamma,q_2$ are
well-posed. The following statements are equivalent
\begin{itemize}
 \item[i.] $\Pq(q_1,q_2) = \complex^I$
 \item[ii.] $\rank \Wq(q_1,q_2) = |I|$
 \item[iii.] $\fq_\alpha(q_1,q_2) \neq 0$ for some $\alpha \in
 \{1,\ldots,|B|^2\}^{|I|}$ with $1 \leq \alpha_1 < \ldots < \alpha_{|I|}
 \leq |B|^2$.
\end{itemize}
\end{lemma}

Hence we have the following.

\begin{lemma}\label{lem:qfanalytic}
Let $\gamma \in \complex_0^E$ and $\zeta$ be such that the Dirichlet
problem for $\gamma,q$ is well-posed for all $q \in \complex_\zeta^I$.
The functions $\fq_\alpha: \complex^I_\zeta \times \complex^I_\zeta \to
\complex$ are analytic for any $\alpha \in \{0,\ldots,|B|^2\}^{|I|}$.
\end{lemma}
\begin{proof}
Let $q_1,q_2 \in \complex^I$.  First note that the entries of the
matrices $((L_\gamma)_{II} + \diag(q_k))^{-1}$, $k=1,2$, are complex
analytic on $q_1$ and $q_2$ when $q_1,q_2\in \complex^I_\zeta$. This is
because the cofactor formula for the inverse guarantees that the entries
of the matrices in question are rational functions in $q_1,q_2$ which
are analytic provided $\det ((L_\gamma)_{II} + \diag(q_k)) \neq 0$,
$k=1,2$.  As in the proof of theorem~\ref{thm:dir}, the condition
$q_1,q_2 \in \complex^I_\zeta$ ensures that $\det ((L_\gamma)_{II} +
\diag(q_k)) \neq 0$, $k=1,2$.  Since $\Wq(q_1,q_2)$ is obtained from
$((L_\gamma)_{II} + \diag(q_k))^{-1}(L_\gamma)_{IB}$, $k=1,2$ by taking
columnwise Hadamard products, each entry of $\Wq(q_1,q_2)$ is also
analytic on $q_1$ and $q_2$ when $q_1,q_2\in \complex^I_\zeta$. Finally
taking the determinant of a matrix with analytic entries is also
analytic. 
\end{proof}

We  can now proceed with the proof of the main result in this section.

\begin{proof}[Proof of Theorem~\ref{thm:bur}]
By lemma~\ref{lem:subq}, the theorem hypothesis means that there is
some multi-index $\beta \in \{1,\ldots,|B|^2\}^{|I|}$ such that
\[
 \fq_\beta(p_1,p_2) \neq 0, ~\text{for some $p_1,p_2 \in
 \complex^{|I|}_\zeta$}.
\] 
Thus $\fq_\beta$ is not identically zero. In fact its zero set restricted
to $\complex^{|I|}_\zeta \times \complex^{|I|}_\zeta$ 
\begin{equation}
Z(\fq_\beta) = \Mcb{ (q_1,q_2) \in \complex^{|I|}_\zeta \times
\complex^{|I|}_\zeta ~ | ~ \fq_\beta(q_1,q_2) = 0 },
\label{eq:zerosetq}
\end{equation}
must be a set of measure zero (see e.g.
\cite{Gunning:1965:AFS}), where we use the Lebesgue measure on
$\complex^{|I|} \times \complex^{|I|}$. By lemma~\ref{lem:subq},
the subset $S$ of $\complex^I_\zeta \times \complex^I_\zeta$ on which $\Pq(q_1,q_2)
\neq \complex^I$ is
\[
 S = \bigcap_{\alpha \in \{1,\ldots,|B|^2\}^{|I|}} Z(\fq_\alpha).
\]
Since $Z(\fq_\beta)$ has measure zero, the set $S \subset Z(\fq_\beta)$ must
have measure zero as well by monotonicity of the Lebesgue measure.
\end{proof}

%%%%%%%%%%%%%%%%%%%%%%%%%%%%%%%%%%%%%%%%%%%%%%%%%%%%%%%%%%%%%%%%%%%%%%%%
\bibliographystyle{siamurl}
\bibliography{schroe}
\end{document}